\documentclass[11pt]{article}
\usepackage{amsmath,amsfonts,amssymb,amsthm,bbm} 
\usepackage{graphicx,graphics,epsfig,subfigure}
\usepackage[margin=1cm]{caption}
\usepackage{graphicx}
\usepackage{booktabs}
\usepackage{threeparttable}
\usepackage{xcolor,color,soul,array,cite,float,version,times}
\usepackage[version=4]{mhchem}
\numberwithin{equation}{section}
\numberwithin{figure}{section}
\numberwithin{table}{section}
\numberwithin{footnote}{section}
\theoremstyle{definition}
\newtheorem{defi}{Definition}[section]

\newtheorem{scheme}{Scheme}[section]
\theoremstyle{plain}
\newtheorem{thm}{Theorem}[section]

\newtheorem{prop}{Property}[section]
\theoremstyle{remark}
\newtheorem{remark}{Remark}[section]

\newcommand{\ben}{\begin{eqnarray}}
\newcommand{\een}{\end{eqnarray}}
\newcommand{\bea}{\begin{array}}
\newcommand{\eea}{\end{array}}
\newcommand{\bes}{\begin{subequations}}
\newcommand{\ees}{\end{subequations}}
\newcommand{\bef}{\begin{figure}[H]}
\newcommand{\eef}{\end{figure}}
\newcommand{\bet}{\begin{tikzpicture}}
\newcommand{\eet}{\end{tikzpicture}}
\newcommand{\beq}{\begin{equation}}
\newcommand{\eeq}{\end{equation}}
\def\bena#1\eena{\begin{eqnarray}\begin{array}{l}#1\end{array}\end{eqnarray}}
\def\besl#1\eesl{\begin{subequations}\begin{equation}#1\end{equation}\end{subequations}}

\def\bx{\mathbf{x}}

\font\tenbi=cmmib10   at 11 pt
\font\sevenbi=cmmib10 at 9pt
\font\fivebi=cmmib7 at 6pt
\newfam\bifam
\textfont\bifam=\tenbi \scriptfont\bifam=\sevenbi
\scriptscriptfont\bifam=\fivebi

\def\bi{\fam\bifam\tenbi}
\font\tendb=msbm10 at 12 pt
\font\sevendb=msbm7
\newfam\dbfam
\textfont\dbfam=\tendb \scriptfont\dbfam=\sevendb

\def\x{{\bi x}}

\begin{document}
\title{A Non-uniform Time-stepping Convex Splitting Scheme for  the Time-fractional  Cahn-Hilliard Equation}
\author{Jun Zhang \thanks{Computational   Mathematics  Research  Center, Guizhou University of Finance and Economics, Guiyang, Guizhou 550025,  China; Department of Mathematics, Guizhou University, Guiyang, Guizhou 550025, China; Email: jzhang@mail.gufe.edu.cn.}
\and Jia Zhao \thanks{Department of Mathematics and Statistics, Utah State University, Logan, UT, 84322, USA; Email: jia.zhao@usu.edu.}
\and JinRong Wang \thanks{Department of Mathematics, Guizhou University, Guiyang, Guizhou 550025, China; School of Mathematical Sciences, Qufu Normal University, Qufu 273165, Shandong, China; Email: jrwang@gzu.edu.cn.}}
\date{\today}
\maketitle

\begin{abstract}
In this paper, a non-uniform time-stepping convex-splitting numerical algorithm for solving the widely used time-fractional Cahn-Hilliard equation is introduced. The proposed numerical scheme employs the $L1^+$ formula for discretizing the time-fractional derivative and a second-order convex-splitting technique to deal with the non-linear term semi-implicitly. Then the pseudospectral method is utilized for spatial discretization. As a result, the fully discrete scheme has several advantages: second-order accurate in time, spectrally accurate in space, uniquely solvable, mass preserving, and unconditionally energy stable. Rigorous proofs are given, along with several numerical results to verify the theoretical results, and to show the accuracy and effectiveness of the proposed scheme.  Also, some interesting phase separation dynamics of the time-fractional Cahn-Hilliard equation has been investigated.

\end{abstract}

\section{Introduction}
The Cahn-Hilliard (CH) equation was originally introduced to describe the process of coarsening dynamics of binary alloys. Ever since, a great deal of peer-reviewed papers have been published investigating different aspects of the Cahn-Hilliard  equation \cite{Blowey1991The,Blowey1992The,Tierra-1,Tierra-2,LiYibao-1,Cheng-1,LiJC}, as well as applying it in many fields. For a mixture with two component, use $\phi$ to label the components: $\phi=1$ to label one component, and $\phi=-1$ to label the other component. Then, the evolution dynamics could be modeled by the Cahn-Hilliard equation
\begin{equation}\label{eq:CH}
\left\{
\begin{array}{l}
\partial_t \phi  + M (-\Delta)(- \varepsilon^2 \Delta \phi+\phi^3-\phi)=0,\ \ \x \in \Omega\subset R^d, \ \ 0< t\leq T,\\
\phi(\x,0)=\phi_0(\x),
\end{array}
\right.
\end{equation}
with proper physically relevant boundary conditions.  Here $\Omega$ is the domain, $d=2,3$ is the spatial dimension, $M>0$ is the mobility parameter and $\varepsilon$ controls the length scale of transition regions. It could be shown that the equation \eqref{eq:CH} is a $H^{-1}$ gradient flow with respect to the free energy
\beq
E(\phi) = \int_\Omega  \Big[ \frac{\varepsilon^2}{2}|\nabla \phi|^2 + \frac{1}{4}(\phi^2-1)^2 \Big] d\x.
\eeq
Some well-known properties include: the total mass of each component is conserved, which could be easily verified by realizing $$
\int_\Omega \phi(\x, t)d\x = \int_\Omega \phi(\x,0)d\x;
$$
and the free energy is non-increasing in time, which could be justified by noticing
\beq
\frac{dE}{dt} = \int_\Omega \frac{\delta E}{\delta \phi} \frac{\delta \phi}{\delta t} d\x = - \int_\Omega M  \Big| \nabla (- \varepsilon^2 \Delta \phi+\phi^3-\phi) \Big|^2 d\x ,
\eeq
given the boundary integral terms vanishes.

Nowadays, the Cahn-Hilliard model has emerged as a classical mathematical physics model in various applications.
However, solving the Cahn-Hilliard equation is non-trivial, given the stiffness introduced by $\varepsilon$ and the nonlinearity in the equation. Many generalized numerical algorithms have been developed to overcome such difficulties, which include the convex splitting method \cite{Elliott1993The, Eyre1998, hu2009stable},
the linear stabilization method \cite{Zhu1999Ch, Zhao2017A,Tierra-1},
the Invariant Energy Quadratization (IEQ)  approach \cite{yang2017numerical}, and the
scalar auxiliary variable (SAV) method \cite{shen2018scalar}.
In addition, there have been extensive works of specific numerical schemes for the Cahn-Hilliard equation, such as second-order finite difference\cite{Sun1995A, Yan2018},  fourth-order finite difference scheme \cite{LiYibao-1}, pseudospectral scheme\cite{ChengK2016}, and mixed finite element method \cite{Feng2004, Diegel2016}.
Moreover, many convex splitting schemes \cite{WiseSINUMA2009,hu2009stable, Wang2011Wsie, Baskaran2013Wsie, ShenSecond2012, Zhen2014A, Chen2016wang, HanA2016} have been applied to various gradient flow models such as  phase-field crystal equation, the modified phase field
crystal equation, epitaxial thin film growth equation,  nonlocal Cahn-Hilliard model, and coupled system of phase field equations.

To manipulate the coarsening dynamics, researchers have proposed some extensions of the classical Cahn-Hilliard equation. For instance, the viscous Cahn-Hilliard equation is postulated by introducing the inertia effect into the dissipation dynamics. One other option of introducing a memory effect (inertia) is to take advantage of the fractional time derivative. Thus, the time-fractional Cahn-Hilliard equation has been introduced, which reads as
\begin{equation}
\left\{
\begin{array}{l}
\partial^{\alpha}_t \phi  + M (-\Delta)(- \varepsilon^2 \Delta \phi+\phi^3-\phi)=0,\ \ \x \in \Omega\subset R^d, \ \ 0< t\leq T,\\
\phi(\x,0)=\phi_0(\x),
\end{array}
\right.
\end{equation}
where $d=2,3$, $\alpha \in (0, 1)$  is the time-fractional order, $M$ is the mobility parameter, and $\varepsilon$ is an artificial parameter controlling the interfacial thickness. Here $\partial^{\alpha}_t \phi$ denotes the classical Caputo derivative defined as
\beq
\partial_t^\alpha \phi(\bx, t) = \frac{1}{\Gamma(1-\alpha)} \int_0^t \frac{\partial \phi(\bx, s)}{\partial s} \frac{ds}{(t-s)^\alpha}, \quad 0 < \alpha < 1.
\eeq
 The time-fractional models are known to maintain the memory effect of materials.

Though with its popularity, the time-fractional Cahn-Hilliard equation has not been derived physically, and its thermodynamic properties are not well understood. In particular, the classical gradient flow problem is known to be derived by energy variation, which makes the derived model satisfy the law of energy dissipation. Whether the fractional-order model also satisfy similar energy dissipation is still an open question.In the meanwhile, there are intensive research activities in understanding the time-fractional phase-field models or gradient flows, and their various anomalous coarsening dynamics in general. With many seminal work has been published, here we only emphasize some relevant research results. For a detailed reference list, please check the papers  \cite{ AM2017,ASMz2017,Liu2018,du2019time, tang2018energy, zhao2019power,chen2019accurate}, and references therein. For the time-fractional phase-field model, Tang et al. \cite{tang2018energy} obtained an energy dissipation property with an integral type.
They proved that the numerical method with the  $L1$ formula for discretizing the fractional time derivative satisfies the energy dissipation property under suitable conditions \cite{lin2007finite}. In our previous work, we considered a series of time-fractional phase-field models, and show numerically that the time-fractional phase-field model follows a scaling law in the coarsening process \cite{zhao2019power,chen2019accurate, Liu2018}.
Furthermore, we observed a linear proportional relationship between the decay rate of energy and fractional derivative $\alpha$, which is in agreement with their counterparts (the integer phase-field models) \cite{Dai2016Computational, Zhu1999Ch}.  This is an exciting phenomenon by revealing the hidden connections between integer and fractional phase-field models. And it sheds light on how the scaling law is implemented.
Liao et al.  constructed a  novel $L1^+$ formula coupled with  IEQ/SAV method to approximate a time-fractional molecular beam epitaxial growth models \cite{ji2019adaptive} and the Allen-Cahn equation \cite{ji2019adaptive-2}.
They proved that the numerical method is energy stable under a discrete integral summation.

For solving phase-field models, adaptivity in time is essential to save computational resources dramatically \cite{QiaoCH}.
When the time mesh is uniform, Tang et al. \cite{tang2018energy} propose a stable time-discrete scheme using the classical $L1$ formula for the fractional time derivative. Their algorithm can't be applied to take into account the weak singularity
of the initial state at $t=0$, and it is not proper to assume that the solution is smooth in the entire closed domain.
For non-smooth initial values, Jin et al. \cite{jin2015analysis} prove that the $L1$ scheme could not achieve $2-\alpha$ order accuracy. Thus, when the time mesh is non-uniform, the $L1$ formula won't guarantee the energy stability anymore.
The main goal of this work is to develop an efficient numerical scheme for the time-fractional Cahn-Hilliard model with non-uniform time steps, which can handle initial singularity and obeys the energy inequality, i.e., energy stable. Our numerical scheme is performed by  utilizing a $L1^+$ formula  \cite{ji2019adaptive} for time-fractional derivative  and a convex splitting technique \cite{WiseSINUMA2009} for non-linear energy function.
We show that our numerical method is uniquely solvable, unconditionally stable, and satisfies the property of energy dissipation.
Several numerical examples are proposed to verify that the numerical scheme can achieve a second-order accuracy in the time direction.
At last,  the coarsening dynamics have been studied.

The rest of the article is organized as follows. In Section 2, we will briefly introduce the time-fractional Cahn-Hilliard equation. Then the $L1^+$ formula and the convex splitting scheme are studied in detail. Some properties of the newly proposed schemes will be introduced, along with detailed proofs. In Section 3, several numerical experiments are performed to demonstrate the effectiveness of the numerical methods. The conclusion of this article is given in the last section.

\section{A Non-Uniform Time Stepping Numerical Algorithm}

For better explanation, we first introduce some notations. Let $(\cdot,\cdot)$ be the $L^2$ inner product. Define for $m\geq 0$
\begin{equation}\label{eq11}
H^{-m}(\Omega)=(H^{-m}(\Omega))^*,\ \ H^{-m}_0(\Omega)=\{\varphi\in  H^{-m}(\Omega)| ( \varphi, 1)_{m}=0 \},
\end{equation}
here  $( \varphi, 1)_{m}$ be the dual product between $H^{m}(\Omega)$ and  $H^{-m}(\Omega)$.
For $u \in L^2_0(\Omega)$, denote $-\Delta^{-1}u=\varphi\in H^1(\Omega)\cap L^2_0(\Omega)$, here $\varphi$ be the solution of
\begin{eqnarray}\label{eq12}
&-\Delta \varphi=u,\ \  \hbox{in} \ \ \Omega,\\  \label{eq13}
&(i)~ \varphi~  \hbox{is periodic ; or} ~(ii) ~ \frac{\partial \varphi}{\partial n}=0, \ \ \hbox{on}\ \ \partial\Omega.
\end{eqnarray}
Since the CH equation  is an $H^{-1}$ gradient flow, we define inner product and norm as
\begin{equation}\label{eq14}
(f,  (-\Delta)^{-1}g)= \Big((-\Delta)^{-\frac{1}{2}} f,  (-\Delta)^{-\frac{1}{2}}g\Big), \ \  \| f\|_{-1}:=\| (-\Delta)^{-\frac{1}{2}} f \|.
\end{equation}

\subsection{Time-fractional Cahn-Hilliard equation}
In this paper we focus on the time-fractional Cahn-Hilliard (TFCH) equation as follows
\begin{equation} \label{eq:TFCH}
\left\{
\begin{array}{l}
\partial^{\alpha}_t \phi + M (-\Delta)(- \varepsilon^2 \Delta \phi+\phi^3-\phi)=0,\ \ \x \in \Omega \ \ 0< t\leq T,\\
\phi(\x,0)=\phi_0(\x),
\end{array}
\right.
\end{equation}
with periodic boundary condition. Here we assume $\Omega \subset R^d$, is a smooth domain, with boundary $\partial \Omega$, and $d=2,3$. $\alpha \in (0, 1)$  is the time-fractional order, $M$ is the mobility parameter, and $\varepsilon$ is an artificial parameter controlling the interfacial thickness. Here $\partial^{\alpha}_t \phi$ denotes the classical Caputo derivative
\begin{equation}\label{eq5}
\partial^{\alpha}_t\phi(\x, t) := \Big( \mathcal{I}_t^{1-\alpha} \partial_t \phi(\x, t) \Big)(t) = \frac{1}{\Gamma(1-\alpha)} \int^t_0  \frac{ \partial_s \phi(\x, s)}{(t-s)^\alpha}\hbox{d}s,
\end{equation}
where $\mathcal{I}_t^{\beta}$,$\beta>0$ is the Riemann-Liouville fractional integration operator defined as
\beq
(\mathcal{I}_t^\beta \phi)(t) = \int_0^t \frac{1}{\Gamma(\beta)} \frac{\phi(\bx, s)}{t^{1-\beta}}  ds,
\eeq
and $\Gamma(\bullet)$ denotes the $\Gamma$-function.

For simplicity, in the rest of this paper, we will assume periodic boundary conditions for the time-fractional Cahn-Hilliard model \eqref{eq:TFCH}. Notice the proposed scheme, along with its properties, also holds for homogeneous Neumann boundary conditions.
It could be verified that the time-fractional Cahn-Hilliard model \eqref{eq:TFCH} has two essential properties.

\begin{prop}[Mass Conservation]
The time-fractional Cahn-Hilliard equation in \eqref{eq:TFCH} preserves the total mass, in the sense of
\beq \label{eq:Mass-Conservation}
\int_\Omega \phi(\x,t) d\x = \int_\Omega \phi(\x,0)d\x.
\eeq
\end{prop}

\begin{proof}
We will verify \eqref{eq:Mass-Conservation} by showing that the solution of \eqref{eq:TFCH} satisfies
$$
\frac{d}{dt} \int_\Omega \phi(\x,t)d\x = 0.
$$
As a matter of fact, this could be verified by noticing
\begin{eqnarray}
0  &=&   \int_\Omega \partial_t^\alpha \phi(\x,t)d\x  \nonumber \\
&= &\frac{1}{\Gamma(1-\alpha)} \int_\Omega \int_0^t \frac{\partial_t \phi(\x,s)}{(t-s)^\alpha}dsdx \nonumber \\
&= & \frac{1}{\Gamma(1-\alpha)} \int_0^t \frac{1}{(t-s)^\alpha}    \Big[  \int_\Omega \partial_t \phi(\x,s)dx \Big] ds, \nonumber%
\end{eqnarray}
Since the domain $\Omega$ is independent of time, we obtain
\beq
\frac{d }{dt} \int_\Omega \phi(\x,t)d\x = \int_\Omega \partial_t \phi(\x,t) d\x = 0.
\eeq
This completes the proof.
\end{proof}

\begin{prop}[Energy Bound] \label{prop:energy-bound}
It satisfies the following energy dissipation law
\begin{equation}\label{eq6}
E(\phi(\x, T))-E(\phi(\x,0))=-\int_{\Omega} A_{\alpha}(\nabla \psi, \nabla \psi ) \hbox{d} \x \leq 0, \quad \forall T \geq 0,
\end{equation}
where
\begin{equation}\label{eq7}
A_{\alpha}(f, g)=\frac{1}{\Gamma(1-\alpha)}\int^T_0 \int^t_0  \frac{f(s)g(s)}{(t-s)^\alpha} \hbox{d}s \hbox{d}t,
\end{equation}
and
$\psi=(-\Delta)^{-1}\phi_t$ is the solution of $- \Delta \psi=\phi_t$ .
Here, the effective free energy of \eqref{eq:TFCH} could be derived as
\begin{equation}\label{eq1}
E(\phi)=\int_{\Omega} \Big( \frac{\varepsilon^2}{2} |\nabla \phi|^2+ \frac{1}{4}(\phi^2-1)^2 \Big)\hbox{d}\x.
\end{equation}
\end{prop}

The detailed proof for the property \ref{prop:energy-bound} could be found in \cite{tang2018energy}. We thus omit it for brevity.

\subsection{Time discretization}
In this section, we introduce the time discretization for the time-fractional Cahn-Hilliard model \eqref{eq:TFCH}. We closely follow the notations in \cite{ji2019adaptive}.

For given  $T>0$ and positive integer $N$, consider the non-uniform graded mesh $0=t_0< t_1< \cdots < t_n< \cdots <t_N=T$, with time step $\tau_n=t_{n}-t_{n-1}, 1\leq n \leq N$.
Given a sequence of  grid functions $\{\phi^n\}^N_{n=1}$, define
\begin{equation}\label{eq23}
\nabla_{\tau}\phi^n:=\phi^n-\phi^{n-1}, \ \ \partial_{\tau}\phi^{n-\frac{1}{2}}:=\nabla_{\tau}\phi^n/\tau_{n}, \ \ \phi^{n-\frac{1}{2}}:=(\phi^n+\phi^{n-1})/2, \ \ 1 \leq n \leq N.
\end{equation}
Denote
\beq  \label{eq:gamma}
\omega_\alpha (t) = \frac{1}{\Gamma(\alpha)} \frac{1}{t^{1-\alpha}},
\eeq
and let $\Pi_{1}\phi(t)$ be the linear interpolant of $\phi(t)$ between $t_{n-1}$ and $t_n$, that is
\begin{equation}\label{eq24}
(\Pi_1 \phi)(t):=\partial_{\tau} \phi^{n-\frac{1}{2}}, \ \ \forall t\in (t_{n-1}, t_n],\ \  1 \leq n \leq N.
\end{equation}

\begin{defi}[$L1$ Formula]
The $L1$ formula for Caputo derivative is defined as
\begin{equation}\label{eq25}
(\partial^\alpha_{\tau} \phi)^n:=\int^{t_n}_{t_0}\omega_{1-\alpha}(t_n-s)\phi'(s)ds=\sum\limits^n_{k=1}a^n_{n-k} \nabla_{\tau}\phi^k£¬
\end{equation}
where $a^n_{n-k}$'s are given as
\begin{equation}\label{eq26}
a^n_{n-k}:=\frac{1}{\tau_k} \int^{t_{k}}_{t_{k-1}}\omega_{1-\alpha}(t_n-s)ds, \ \ 1\leq k\leq n.
\end{equation}
\end{defi}

The $L1$ formula has several advantages. In particular, it could be easily seen that
The coefficient $a^n_{n-k}$ satisfies the following properties \cite{liao2018sharp,liao2019unconditional}
\begin{equation}\label{eq27}
a^n_{n-k}>0, \ \ \ \ a^n_{n-k-1}\geq a^n_{n-k}, \ \ \ \ 1\leq k \leq n-1.
\end{equation}
In particular, for uniform meshes, i.e., $t_n = \frac{n}{N}T$,  we have
\begin{equation}\label{eq29}
a_{n-k}^n=\frac{1}{\tau^\alpha}\Big[ \omega_{2-\alpha}(n-k+1)-\omega_{2-\alpha}(n-k)\Big],\ \ \ \   1 \leq k \leq n.
\end{equation}
The coefficients in \eqref{eq29} also satisfy the follow inequality.
\begin{prop}[Discrete Convolution Formula]
For any real sequence $\{\phi^i \}_{i=1}^n$, it holds
\begin{equation}\label{eq28}
\sum\limits^n_{k=1}   \sum\limits^k_{j=1}a_{k-j}^n \phi^k\phi^j \geq0.
\end{equation}
\end{prop}

The discrete convolution formula has been utilized in \cite{tang2018energy} to prove the energy dissipation property.
However, it seems difficult to obtain similar semi-positive definite properties for non-uniform time grids. It turns out the $L1^{+}$ formula introduced in \cite{ji2019adaptive} will overcome such difficulties.
\begin{defi}[$L1^{+}$ Formula ( see \cite{ji2019adaptive})]
The  $L1^+$ formula of the  Caputo derivative  at $t^{n-\frac{1}{2}}$ is given as
\begin{eqnarray}
(\partial^\alpha_{\tau} \phi)^{n-\frac{1}{2}}&:=&\frac{1}{\tau_n} \int^{t_n}_{t_{n-1}} \int^t_0 \omega_{1-\alpha}(t-s) (\Pi_1 \phi)'(s)\hbox{d} s \hbox{d}t  \nonumber \\
&=& \sum\limits^n_{k=1} \overline{a}^{n}_{n-k} \nabla_{\tau} \phi^k, \ \   n \geq 1, \label{eq30}
\end{eqnarray}
where $\overline{a}^{n}_{n-k}$ are  defined by
\begin{equation}\label{eq31}
\overline{a}^{n}_{n-k}:=\frac{1}{\tau_k \tau_n} \int^{t_n}_{t_{n-1}} \int^{ \min\{t, t_k\}}_{t_{k-1}}\omega_{1-\alpha}(t-s) \hbox{d}s \hbox{d}t, \ \ 1\leq k \leq n.
\end{equation}
\end{defi}
Obviously, the discrete convolution kernels  $\overline{a}^{n}_{n-k}$ is positive. In fact, they have many good properties. See \cite{ji2019adaptive} for details. Here we only emphasis the following one.
\begin{prop}[Discrete Convolution Formula]
For any real sequence $\{\phi^i \}_{i=1}^n$, it holds
\begin{equation}\
\sum\limits^n_{k=1}   \sum\limits^k_{j=1}\overline{a}^k_{k-j}\phi^k \phi^j \geq 0.
\end{equation}
\end{prop}

For the nonlinear terms, we utilize a convex splitting strategy \cite{WiseSINUMA2009} to introduce an explicit-implicit temporal discretization. Overall, the semi-discrete scheme in time is proposed as
\begin{scheme}[Non-uniform time marching scheme] \label{scheme:semi}
Set $\phi^{-1}=\phi^0$. For given  $T>0$ and positive integer $N$, consider the non-uniform graded mesh $0=t_0< t_1< \cdots < t_m< \cdots < t_N=T$, with time step $\tau_m=t_{m}-t_{m-1}, 1\leq m \leq N$. After we obtain $\phi^i$, $i \leq n-1$ with  $n \geq 1$, we can get $\phi^n$ via the following scheme
\beq  \label{eq36}
\left\{
\bea{l}
(\partial_\tau^\alpha \phi)^{n-\frac{1}{2}} = M \Delta \mu^{n-\frac{1}{2}}, \\
\mu^{n-\frac{1}{2}} = -\frac{\varepsilon^2}{2}  \Big( \Delta  \phi^{n-1}+ \Delta \phi^n \Big) + \frac{1}{4} \Big( (\phi^{n-1}+\phi^n) ((\phi^{n})^2 +(\phi^{n-1})^2)  \Big) -(\frac{3}{2}\phi^{n-1} - \frac{1}{2}\phi^{n-2}),
\eea
\right.
\eeq
where the formula for the temporal fractional derivative is given in \eqref{eq30}.
\addtocounter{scheme}{-1}
\end{scheme}

%

\begin{remark}
It is worth mentioning that the above time-discrete scheme \eqref{eq36} works for any non-uniform grid.
Therefore, the solution singularity near the initial time can be effectively handled by proposing proper non-uniform time meshes or smaller meshes in general. Also, in a certain time regime when dynamics evolve slow, larger time steps could be used to reduce the computational time significantly.  Note that, by using $\phi^{-1}=\phi^0$, the local truncation error for the initial step is second-order, which implies that the overall method is globally second-order accurate in time. Meanwhile, numerical tests will verify that the time-discrete scheme can achieve second-order accuracy for certain examples.
\end{remark}

{\color{blue} 
The numerical scheme \eqref{eq36} and the full discrete scheme \eqref{eq:full-discrete-scheme} are second-order accurate for $\phi^n$ when $\tau_n = \tau_{n-1}$, and first-order accurate for $\phi^n$ when $\tau_n \neq \tau_{n-1}$. As an improvement, we can easily achieve the second-order accuracy of the non-uniform time marching scheme by replacing
$-(\frac{3}{2}\phi^{n-1} - \frac{1}{2}\phi^{n-2})$  with $-(\phi^{n-1}+\nabla_\tau \phi^{n-1}/2 \rho_{n-1})$ in equation \eqref{eq36}. Therefore, the improved scheme reads as below. 

\begin{scheme}[Non-uniform time marching scheme] 
Set $\phi^{-1}=\phi^0$. For given  $T>0$ and positive integer $N$, consider the non-uniform graded mesh $0=t_0< t_1< \cdots < t_m< \cdots < t_N=T$, with time step $\tau_m=t_{m}-t_{m-1}, 1\leq m \leq N$. After we obtain $\phi^i$, $i \leq n-1$ with  $n \geq 1$, we can get $\phi^n$ via the following scheme
\begin{equation}
\left\{
\bea{l}
(\partial_\tau^\alpha \phi)^{n-\frac{1}{2}} = M \Delta \mu^{n-\frac{1}{2}}, \\
\mu^{n-\frac{1}{2}} = -\frac{\varepsilon^2}{2}  \Big( \Delta  \phi^{n-1}+ \Delta \phi^n \Big) + \frac{1}{4} \Big( (\phi^{n-1}+\phi^n) ((\phi^{n})^2 +(\phi^{n-1})^2)  \Big) -(\phi^{n-1}+\nabla_\tau \phi^{n-1}/2 \rho_{n-1}),
\eea
\right.
\tag{\ref{eq36}}
\end{equation}
where the formula for the temporal fractional derivative is given in \eqref{eq30}. Here the local time-step ratio $\rho_{n}:=\tau_n/\tau_{n+1}$.
\end{scheme}



It is worth noting that such changes will not affect the subsequent proofs. The only difference for the proof of energy stability is an introduction of the time step ratio $\rho_{n-1}$ restriction, that is $\rho_{n-1}\geq 0.5$, i.e. $\frac{\tau_{n}}{\tau_{n-1}} \leq  2$. And the only difference in the proof of the energy stability is in \eqref{eq:explicit-treatment}. 

As a matter of fact, taking the inner product of the second equation in the scheme with  $\nabla_\tau \phi^n$, we obtain
\begin{align*}
(\mu^{n-\frac{1}{2}},\nabla_\tau \phi^n)=\frac{\varepsilon^2}{2}(\|\nabla \phi^n\|^2-\|\nabla \phi^{n-1}\|^2)+\frac{1}{4}(\| \phi^n\|^4- \|\phi^{n-1} \|^4)-\Big(\phi^{n-1}+\nabla_\tau \phi^{n-1}/2 \rho_{n-1}, \nabla_\tau \phi^n\Big).
\end{align*}
We denote  $\frac{1}{\rho_{n-1}}=1+a$, and use
the following  two identities
\begin{align*}
2a(a-b)=&a^2-b^2+(a-b)^2,\\
2(a-b)(b-c)=&(a-b)^2+(b-c)^2-(a-2b+c)^2.
\end{align*}
Given $\rho_{n-1} \geq \frac{1}{2}$, we have $|a|\leq 1$. Notice the fact 
\begin{align*}
-\Big(\phi^{n-1}+\nabla_\tau \phi^{n-1}/2 \rho_{n-1}, \nabla_\tau \phi^n\Big)=&-\Big(  (1+\frac{1}{2\rho_{n-1}})\phi^{n-1}-\frac{1}{2\rho_{n-1}}\phi^{n-2},  \phi^n-\phi^{n-1}\Big).
\end{align*}
 We can find
\begin{equation*}
\bea{l}
-\Big(\phi^{n-1}+\nabla_\tau \phi^{n-1}/2 \rho_{n-1}, \nabla_\tau \phi^n\Big) \\
=-\Big( (\frac{3}{2}+a)\phi^{n-1}-(\frac{1}{2}+a)\phi^{n-2},  \phi^n-\phi^{n-1}\Big)\\
=-\Big(\frac{3}{2}\phi^{n-1}-\frac{1}{2}\phi^{n-2},\phi^n-\phi^{n-1}\Big)-a\Big(\phi^{n-1}-\phi^{n-2}, \phi^n-\phi^{n-1}\Big)\\
=\frac{1}{2}\Big[\|\phi^{n-1}\|^2-\|\phi^{n}\|^2+\frac{1}{2}\|\phi^n-\phi^{n-1} \|^2\Big]\\
-\frac{1}{4}\Big[\|\phi^n-\phi^{n-1}\|^2 +\|\phi^{n-1}-\phi^{n-2}\|^2-\|\phi^n-2\phi^{n-1}+\phi^{n-2} \|^2\Big]-a\Big(\phi^{n-1}-\phi^{n-2}, \phi^n-\phi^{n-1}\Big)\\
\geq  \frac{1}{2}\| \phi^{n-1}\|^2-\frac{1}{2}\| \phi^{n}\|^2+\frac{1}{4} \|\phi^n-\phi^{n-1}\|^2- \frac{1}{4} \|\phi^{n-1}-\phi^{n-2}\|^2\\
+\frac{1}{4}\|\phi^n-2\phi^{n-1}+\phi^{n-2} \|^2-\frac{|a|}{4} \|\phi^n-2\phi^{n-1}+\phi^{n-2} \|^2\\
\geq -\frac{1}{2}\| \phi^{n}\|^2+\frac{1}{2}\| \phi^{n-1}\|^2+\frac{1}{4} \|\phi^n-\phi^{n-1}\|^2- \frac{1}{4} \|\phi^{n-1}-\phi^{n-2}\|^2.
\eea
\end{equation*}
}

\subsection{Properties of the semi-discrete scheme}
For the proposed scheme \eqref{eq36}, it satisfies several properties. 
First of all, it could be verified that
\begin{thm}[Existence and Uniqueness]
There exists a unique solution at each time step for the proposed scheme in \eqref{eq36}.
\end{thm}

\begin{proof}
The proof is based on a convexity argument.  In fact, we define the minimization problem
\beq \label{eq:convex-optimization}
\min_{\phi} G(\phi),
\eeq
where
\beq
G(\phi) = \frac{1}{4}(-\varepsilon \Delta \phi, \phi)_{L^2(\Omega)}^2 + ( h_1(\phi), 1)_{L^2(\Omega)}^2 + (h_2(\phi), 1)_{H^{-1}(\Omega)}^2,
\eeq
with $h_1(\phi)$ and $h_2(\phi)$ defined as
\beq
\bea{l}
h_1(\phi) = \frac{\phi^4}{4} + \frac{\phi^3}{3}\phi^{n-1} + \frac{\phi^2}{2} (\phi^{n-1})^2 + \phi (\phi^{n-1})^3  -  \phi ( \frac{\varepsilon^2}{2}\Delta \phi^{n-1} + \frac{3}{2}\phi^{n-1}-\frac{1}{2}\phi^{n-2} ), \\
h_2(\phi) =  \frac{1}{2} \overline{a}_0^n \phi^2 +  \Big( \sum_{k=2}^n \overline{a}^n_{n-k} \nabla_\tau \phi^k - \overline{a}_0^n \phi^{n-1} \Big) \phi.
\eea
\eeq
It could be easily verified that the target functional $G(\phi)$ is convex with respect to $\phi$ \cite{WiseSINUMA2009}. Then, there is a unique solution for \eqref{eq:convex-optimization}.

Notice the solution to \eqref{eq36} is equivalent to the solution minimizing the convex functional. Thus,  there exists a unique solution for \eqref{eq36}.
\end{proof}

\begin{thm}[Mass Conservation]
The time-discrete scheme \eqref{eq36} preserves the total mass, i.e.
\beq  \label{eq:mass-conservation}
\int_\Omega \phi^n d\x = \int_\Omega \phi^0 d\x, \qquad \forall n \geq 1.
\eeq
\end{thm}
\begin{proof}
This could be proved by induction by following the similar idea as in \cite{ji2019adaptive-2}. First of all, it could be easily shown $\int_\Omega \phi^1 d\x = \int_\Omega \phi^0d\x$. Then, using the induction, assuming it holds $\int_\Omega \phi^k d\x = \int_\Omega \phi^0 d\x$, $\forall  k < n$, we have
\beq
(\overline{a}_0^n \nabla_{\tau} \phi^n, 1) = (  (\partial_{\tau}^\alpha \phi)^{n-\frac{1}{2}}, 1) = (M \nabla \mu^{n-\frac{1}{2}}, 1) = 0,
\eeq
i.e.
\beq
\int_\Omega \phi^n d\x = \int_\Omega \phi^{n-1}d\x = \int_\Omega \phi^0 d\x.
\eeq
Thus, the scheme \eqref{eq36} conserves the total mass. It completes the proof.
\end{proof}

\begin{thm}[Energy Stability]
The time-discrete scheme  \eqref{eq36} is unconditionally energy stable, and it follows the energy dissipation law as
\begin{equation}\label{eq37}
E(\phi^n)\leq  E(\phi^0),\ \   1 \leq n \leq N.
\end{equation}
\end{thm}

\begin{proof}
From \cite{mclean1996discretization,mclean2007second}, we know the weakly singular kernel $\omega_{1-\alpha}$ is positive semi-define, that is
\begin{eqnarray}\label{eq19}
\mathcal{I}^1_t(w \mathcal{I}^{1-\alpha}_t w)(t) &=&\int^t_0 w(\eta) \hbox{d} \eta \int^\eta_0 \omega_{1-\alpha}(\eta-s)w(s) \hbox{d}s \nonumber\\
&=&\frac{1}{2}\int^t_0 \int^t_0 \omega_{1-\alpha}(|\eta-s|)w(s) w(\eta)\hbox{d}\eta \hbox{d}s \nonumber \\
&\geq& 0.
\end{eqnarray}

Taking $w=\Pi_1 \phi$ in \eqref{eq19}, we know  the the non-uniform $L1^+$ formula
\begin{equation}\label{eq32}
(\partial^\alpha_{\tau} \phi)^{n-\frac{1}{2}} \approx  \frac{1}{\tau_n} \int^{t_n}_{t_{n-1}} (\partial^\alpha_t \phi)(t) \hbox{d}t.
\end{equation}
can ensure that the discrete convolution satisfies semi-positive characterization
\begin{eqnarray}\label{eq33}
\sum\limits^n_{k=1} \nabla_{\tau} \phi^k  (\partial^\alpha_{\tau} \phi)^{k-\frac{1}{2}} &=&
\int^t_{t_0} (\Pi_1 \phi)'(t)\int^t_0 \omega_{1-\alpha} (t-s)  (\Pi_1 \phi)'(s) \hbox{d}s \hbox{d}t  \nonumber\\
&=&\mathcal{I}^1_t\Big(  (\Pi_1 \phi)' \mathcal{I}^{1-\alpha}_t (\Pi_1 \phi)' \Big)(t_n) \nonumber \\
&\geq &0.
\end{eqnarray}
Combining with  \eqref{eq31} and \eqref{eq33},  for any real grid function  $\{w_k\}^n_{k=1}$, we find
\begin{equation}\label{eq34}
\sum\limits^n_{k=1}w_k \sum\limits^k_{j=1} \overline{a}^{n}_{n-k}w_j \geq 0.
\end{equation}

Taking the inner product of  first equation in \eqref{eq36} with  $-\Delta^{-1} \nabla_\tau \phi^{n}$,  we obtain
\begin{equation}\label{eq38}
\Big((\partial^\alpha_{\tau} \phi)^{n-\frac{1}{2}},  (-\Delta)^{-1} \nabla_\tau \phi^{n}\Big)=-(\mu^{n-\frac{1}{2}}, \nabla_\tau \phi^n).
\end{equation}
Taking the inner product of  second equation in \eqref{eq36} with  $\nabla_\tau \phi^n$,  we obtain
\begin{equation}\label{eq39}
(\mu^{n-\frac{1}{2}}, \nabla_\tau \phi^n)=\frac{\varepsilon^2}{2}(\| \nabla \phi^n\|^2-\| \nabla \phi^{n-1}\|^2)+ \frac{1}{4}(\|  \phi^n\|^4-\|  \phi^{n-1}\|^4)
- (\frac{3}{2}\phi^{n-1}-\frac{1}{2}\phi^{n-2}, \nabla_\tau \phi^n).
\end{equation}

Notice the fact
\beq \label{eq:explicit-treatment}
\bea{l}
-(\frac{3}{2}\phi^{n-1}-\frac{1}{2}\phi^{n-2},\nabla_\tau \phi^n) \\
= -\frac{1}{2} \| \phi^n \|^2 +  \frac{1}{2} \| \phi^{n-1}\|^2 + \frac{1}{2} \| \phi^n-\phi^{n-1} \|^2 - \frac{1}{2} (\phi^n - \phi^{n-1}, \phi^{n-1}- \phi^{n-2}) \\
\geq -\frac{1}{2} \| \phi^n \|^2 +  \frac{1}{2} \| \phi^{n-1}\|^2 + \frac{1}{2} \| \phi^n-\phi^{n-1}\|^2 - \frac{1}{4} \| \phi^n - \phi^{n-1}\|^2 -\frac{1}{4} \| \phi^{n-1}- \phi^{n-2}\|^2 \\
= -\frac{1}{2} \| \phi^n \|^2 +  \frac{1}{2} \| \phi^{n-1}\|^2 + \frac{1}{4} \| \phi^n-\phi^{n-1}\|^2 -\frac{1}{4} \| \phi^{n-1}- \phi^{n-2}\|^2. \\
\eea
\eeq


Summing the above equations up, we obtain
\begin{equation}\label{eq41}
E(\phi^k)-E(\phi^{k-1}) \leq - \Big( (-\Delta)^{-\frac{1}{2}}(\partial^\alpha_{\tau} \phi)^{k-\frac{1}{2}}, (-\Delta)^{-\frac{1}{2}} \nabla_\tau \phi^{k}\Big) - \frac{1}{4} \| \phi^k-\phi^{k-1}\|^2 + \frac{1}{4} \| \phi^{k-1}- \phi^{k-2}\|^2 ,
\end{equation}
for any $2 \leq k \leq n$.
Using the fact $\phi^0  =\phi^{-1}$, we have
\beq
E(\phi^1)-E(\phi^0) \leq  - \Big( (-\Delta)^{-\frac{1}{2}}(\partial^\alpha_{\tau} \phi)^{\frac{1}{2}}, (-\Delta)^{-\frac{1}{2}} \nabla_\tau \phi^{1}\Big) - \frac{1}{4} \| \phi^1-\phi^0\|^2.
\eeq

By summing up  for $k=1,\cdots n$,  and  applying  \eqref{eq19}, we have
\begin{eqnarray}\label{eq42}
E(\phi^n)-E(\phi^0) &\leq &- \int_{\Omega} \mathcal{I}^1_t\Big[(  (-\Delta)^{-\frac{1}{2}}\Pi_1 \phi)'\mathcal{I}^{1-\alpha}_t (  (-\Delta)^{-\frac{1}{2}}\Pi_1 \phi)' \Big](t^n) \hbox{d} \x  -\frac{1}{4} \| \phi^n-\phi^{n-1}\|^2 \nonumber\\
&\leq & 0.
\end{eqnarray}
It completes the proof.
\end{proof}

\begin{remark}
We obtained an energy stable numerical scheme for the time-fractional Cahn-Hilliard model. However, how to prove  $E(\phi^n)\leq  E(\phi^{n-1})$ for time-fractional gradient flow problem is still an open question. In addition, the error analysis of the numerical scheme  is meaningful and challenging work.
\end{remark}

\subsection{Spatial discretization}
For the spatial discretization, we use the Fourier Pseudo-spectral method.  To make this paper self-consistent, we introduce a few notations. For more details, interested readers can refer to our previous work \cite{Chen&Zhao&GongCICP2019,Gong&Zhao&WangACM}.

Consider a rectangular domain $\Omega = [0,L_x]\times[0,L_y]$ with $L_x$ and $L_y$ the lengths of each sides. We We partition the domain ? with uniform meshes, with mesh size $h_x=L_x/N_x, h_y=L_y/N_y$, where $N_x$ and $N_y$ be two positive even integers. Thus, the discrete domain is denoted as
 $$\Omega_{h} =
\left\{(x_{j},y_{k})|x_{j} = j h_x, y_{k} = kh_y,~0\leq j\leq
N_{x}-1,0\leq k\leq N_{y}-1\right\}.$$
Also, we introduce $V_{h} = \big\{u|u=\{u_{j,k}|(x_{j},y_{k})\in
\Omega_{h}\} \big\}$ as the space of grid functions on $\Omega_{h}$. 



In order to derive the algorithm conveniently, we denote discrete gradient operator and the discrete Laplace operator
as $\nabla_h$ and $\Delta_h$. Applying the Fourier pseudospectral method in space to the semi-discrete scheme (2.24), we obtain the following fully discrete scheme

\begin{scheme}[Full Discrete Scheme]  \label{fScheme:FullDiscrete}
Give the initial condition $\phi^0 = \phi_0\in V_h$. Set $\phi^{-1}=\phi^0$.   After we obtained $\phi^i$, $i\leq n-1$, with $n \geq 1$, we can update $\phi^n \in V_h$ via
\beq \label{eq:full-discrete-scheme}
\left\{
\bea{l}
(\partial_\tau^\alpha \phi)^{n-\frac{1}{2}} = M \Delta_h \mu^{n-\frac{1}{2}}, \\
\mu^{n-\frac{1}{2}} = -\frac{\varepsilon^2}{2}  \Big( \Delta_h  \phi^{n-1}+ \Delta_h \phi^n \Big) + \frac{1}{4} \Big( (\phi^{n-1}+\phi^n) ((\phi^{n})^2 +(\phi^{n-1})^2)  \Big) -(\frac{3}{2}\phi^{n-1} - \frac{1}{2}\phi^{n-2}),
\eea
\right.
\eeq
where the formula for the temporal fractional derivative is given in \eqref{eq30}.
\end{scheme}

Here we emphasize that the fully discrete scheme also satisfies the three properties: solution existence and uniqueness, mass conservation, and energy dissipation in the full discrete sense. We omit the details as the proofs for the fully discrete scheme are similar to those for the semi-discrete scheme.

\section{Numerical examples}
In this section, the fully discrete numerical scheme in \eqref{eq:full-discrete-scheme} is implemented. Then, we conduct a time-step refinement test to show the second-order temporal accuracy of the proposed scheme. Afterward, several numerical examples are shown to investigate the effects of fractional order $\alpha$ and initial profiles on phase separation dynamics.

\subsection{Convergence tests}
First of all, we perform a time convergence test to demonstrate its order of accuracy. Here we choose the domain $\Omega=[0 \,\,\, 1]^2$, and the parameters $\varepsilon=0.01$, $M=10\times 10^{-5}$. We use $128^2$ meshes, and choose the initial profile for $\phi$ as
\beq
\phi(x,y,t=0) = \tanh \Big( \frac{1}{\sqrt{2}\epsilon}  [\sqrt{(x-x_0)^2 + (y-y_0)^2}-0.25 - \frac{1+\cos(4 \arctan\frac{y-y_0}{x -x_0})}{16}] \Big),
\eeq
with $x_0=y_0=0.5$. For different fractional order $\alpha$, the numerical solution at $T=0.01$ using different time steps $\Delta t = \frac{10^{_3}}{2^n}$, $n in \mathbb{N}$ are calculated.  Since there is generally hard to find the exact solution,  we define the reference 'exact' solution $\phi_{ij}^{ref}$ by the result with its nearest finer time step. the discrete $l^2$ norm of numerical errors $e_{ij}:=\phi_{ij}-\phi_{ij}^{ref}$ are summarized in Table \ref{tab:convergence}. It is observed that 2nd order accuracy in time is reached for the testing problem.
\begin{table}[H]
\centering
\caption{The $l^2$ norm of numerical errors  for $\phi$  at time $T=0.01$ for time fractional Cahn-Hilliard equation with various fractional order $\alpha$. They are computed by the proposed scheme in \eqref{eq:full-discrete-scheme} using various temporal step sizes.}
\label{tab:convergence}
\small
\begin{tabular}{|c|c|c|c|c|c|c|}
\hline
$\delta t$ & $L^2$ Error ($\alpha=0.8$) & Order  & $L^2$ Error ($\alpha=0.5$) & Order  & $L^2$ Error ($\alpha=0.35$) & Order \\
\hline
0.001  & $7.096 \times 10^{-6}$  &  & $4.357  \times 10^{-5}$ &  & $1.227\times 10^{-4}$ & \\
\hline
0.0005 & $1.854  \times 10^{-6}$ & 1.94 & $1.094 \times 10^{-5}$  & 1.99 & $1.603 \times 10^{-5}$ & 2.94 \\
\hline
0.00025 & $4.811  \times 10^{-7}$ & 1.95 & $2.725 \times 10^{-6}$ & 2.01 & $3.970 \times 10^{-6}$ & 2.01 \\
\hline
0.000125 & $1.241  \times 10^{-7}$ & 1.96& $6795 \times 10^{-7}$ & 2.00 & $9.897 \times 10^{-7}$ &  2.00 \\
\hline
0.0000625 & $3.186 \times 10^{-8}$ & 1.97 & $1.696 \times 10^{-7}$ & 2.00 & $2.472 \times 10^{-7}$ & 2.00 \\
\hline
0.00003125 & $8.146 \times 10^{-9}$ & 1.97 & $4.234 \times 10^{-8}$ & 2.00  & $6.186 \times 10^{-8}$ & 1.99 \\
\hline
\end{tabular}
\end{table}

\subsection{Coarsening dynamics with various fractional order $\alpha$}
Next, we study the coarsening dynamics of the time-fractional Cahn-Hilliard equation with the newly proposed scheme. In this case, we choose the domain $\Omega=[0 \,\,\, 1]^2$, parameters $\varepsilon=0.01$, $M = 0.1$. And we use $128^2$ meshes.  For this case, we use a randomly generated initial condition for $\phi$ as
\beq
\phi(x,y, t=0)=10^{-3} rand(x,y),
\eeq
with $rand(x,y) \in [-1 \,\,\, 1]$ random numbers with uniform distribution. To reduce computational time without loosing accuracy, we use an adaptive time marching strategy. Defining $\Delta t_{\min}=10^{-4}$, $\Delta t_{\max}=10^{-1}$ and $\beta=10^{7}$,  the time step $\Delta t_{n+1}$ is chosen by following the formula
\beq \label{eq:time-adaptive}
\Delta t_{n+1} =
\left\{
\bea{l}
\Delta t_{\min}, \quad  n < K, \\
\max  ( \Delta t_{\min}, \frac{ \Delta  t_{\max}}{ \sqrt{1 + \beta \Big| \frac{E^{n}-E^{n-1}}{\Delta t_n} \Big|^2 } } ), \quad n \geq K, \
\eea
\right.
\eeq
Here we pick $K=100$.
The profiles of $\phi$ at different times are summarized in Figure \ref{fig:Coarsening-phi}. We observe either case has similar coarsening dynamics.
From the numerical simulations in \ref{fig:Coarsening-phi}, we also observe that the coarsening dynamics with smaller fractional-order $\alpha$ is faster than that with bigger fractional-order $\alpha$ in the time range $[0 \,\,\, 1]$.

\begin{figure}[H]
\center
\subfigure[$\alpha=0.35$]{
\includegraphics[width=0.15\textwidth]{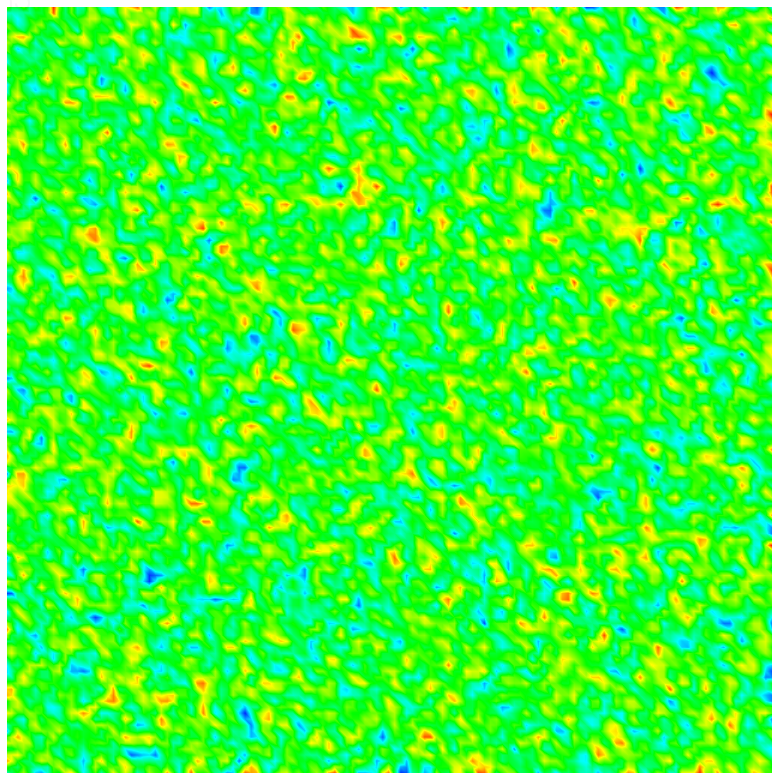}
\includegraphics[width=0.15\textwidth]{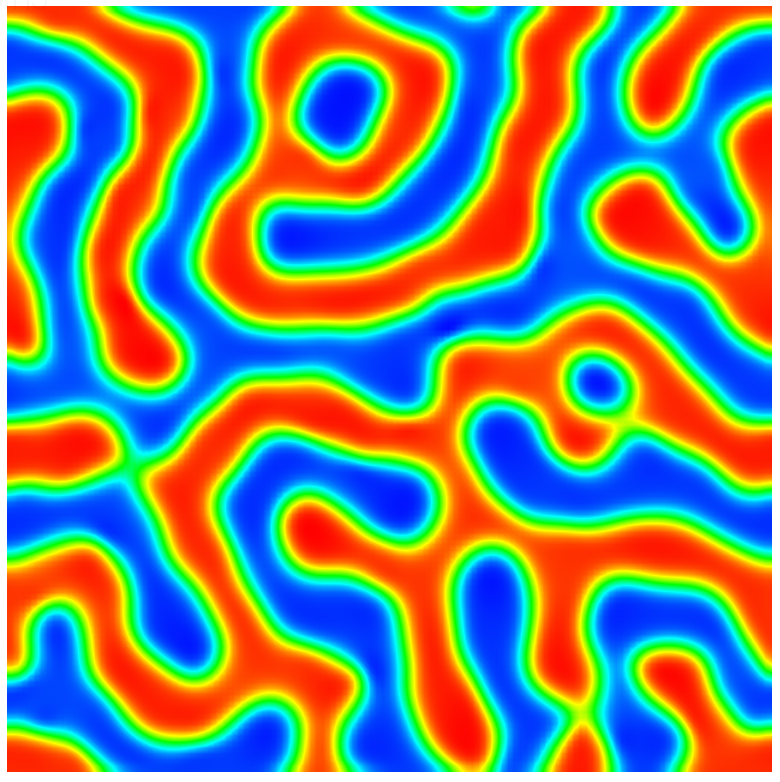}
\includegraphics[width=0.15\textwidth]{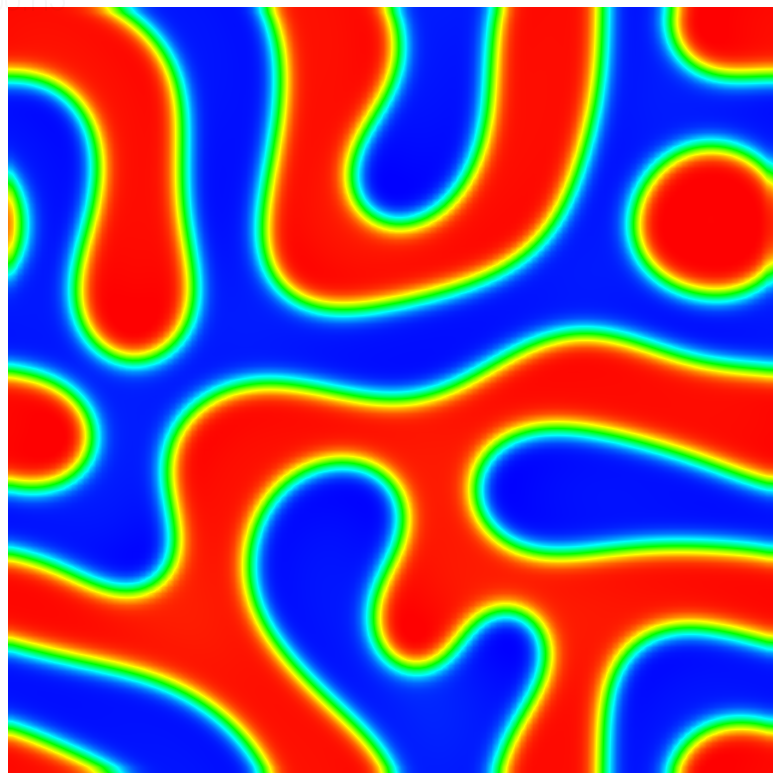}
\includegraphics[width=0.15\textwidth]{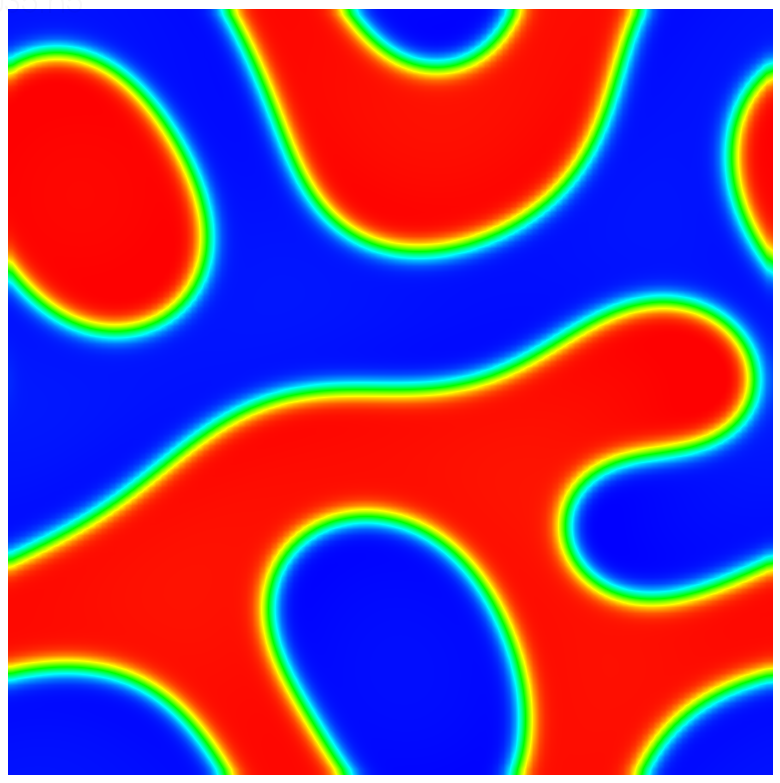}
}

\subfigure[$\alpha=0.5$]{
\includegraphics[width=0.15\textwidth]{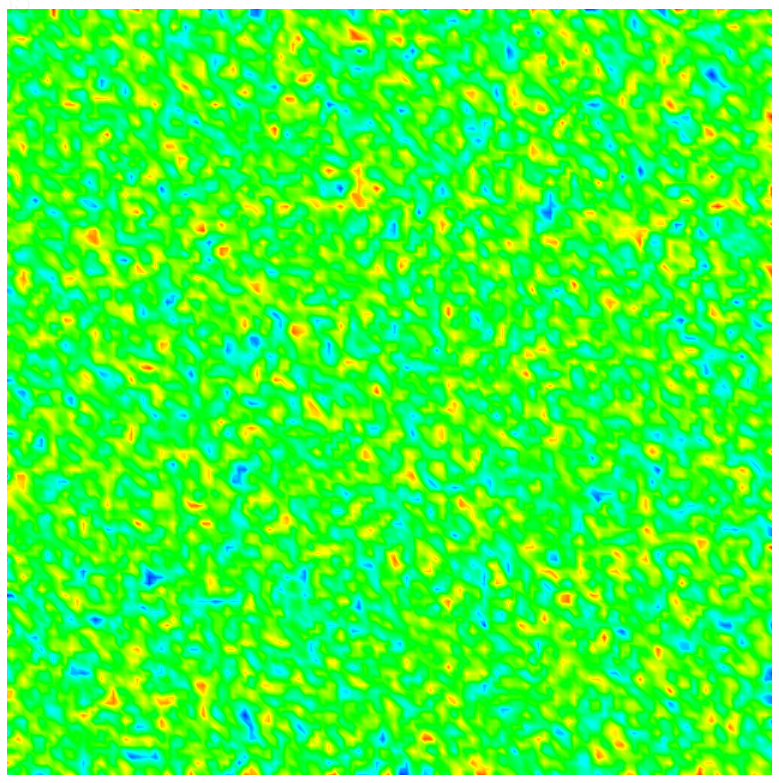}
\includegraphics[width=0.15\textwidth]{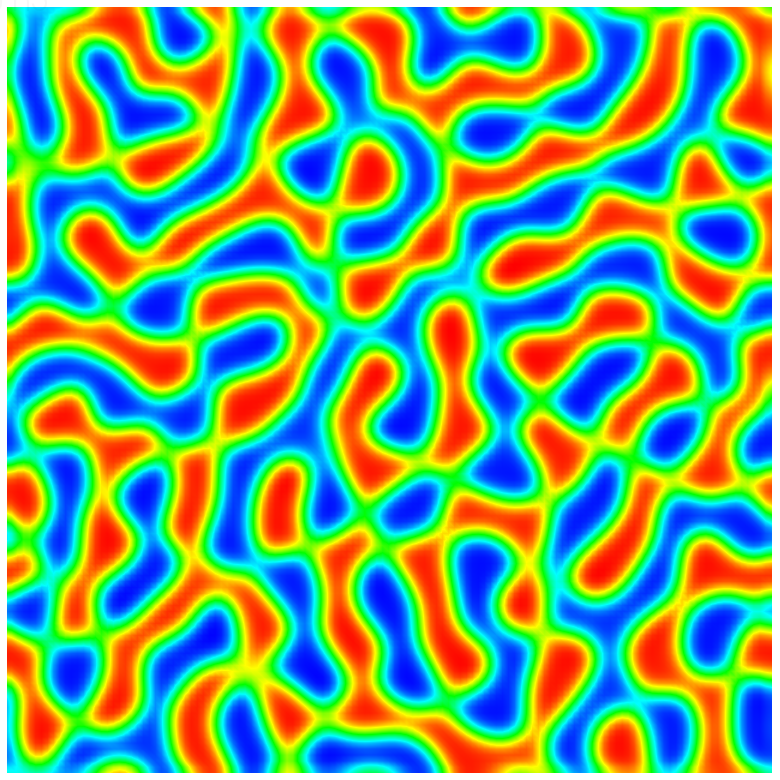}
\includegraphics[width=0.15\textwidth]{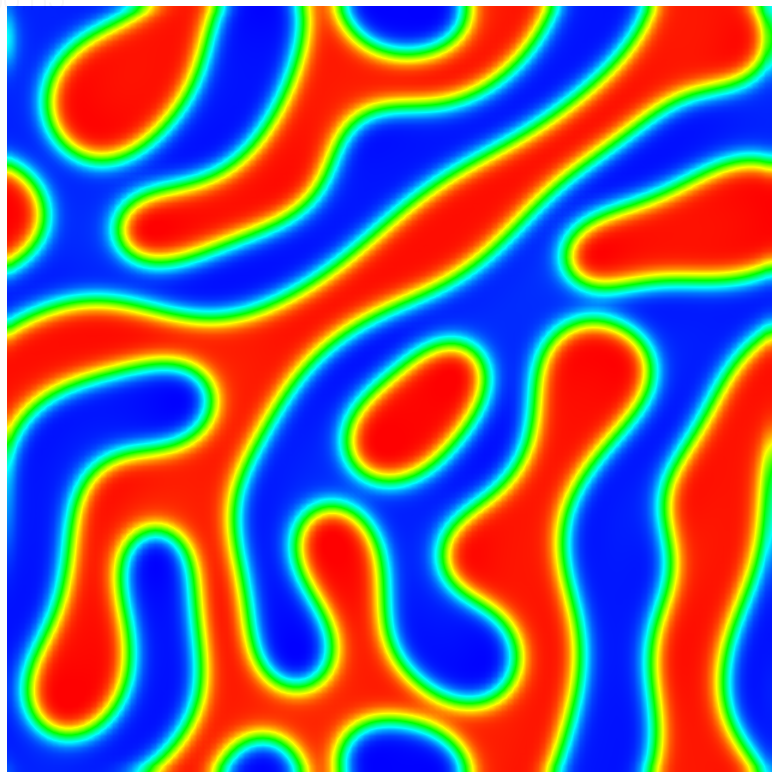}
\includegraphics[width=0.15\textwidth]{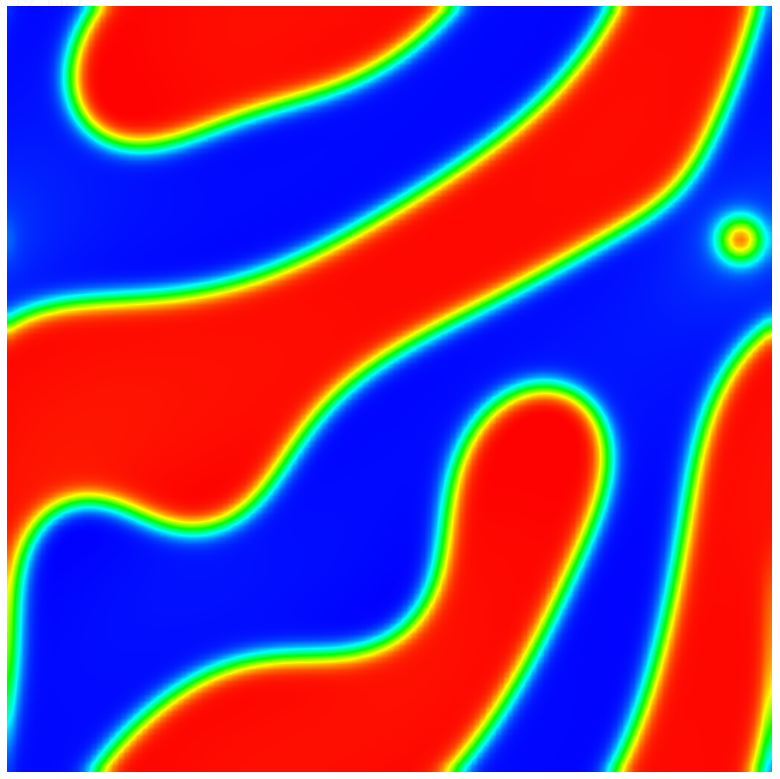}
}

\subfigure[$\alpha=0.8$]{
\includegraphics[width=0.15\textwidth]{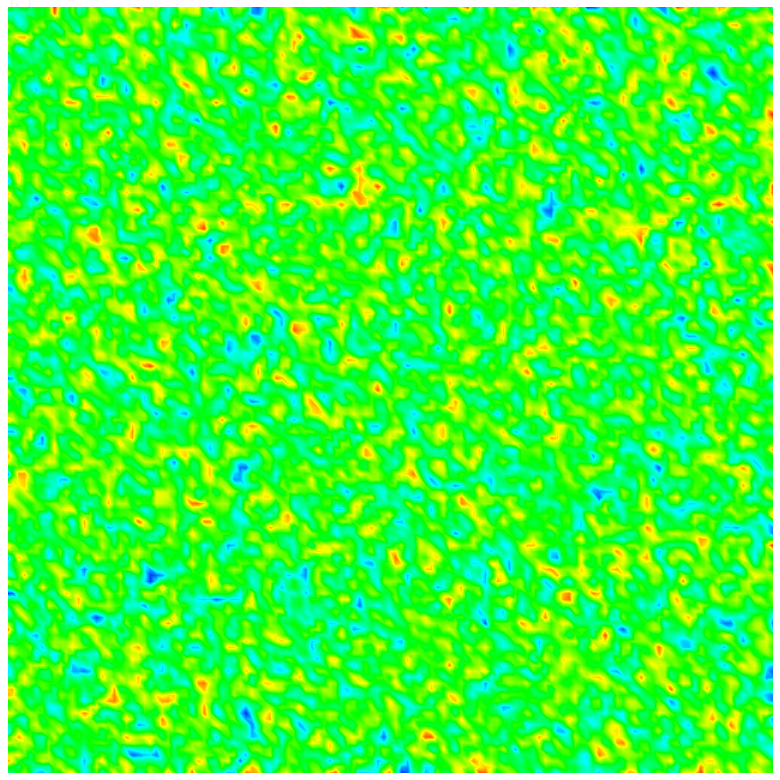}
\includegraphics[width=0.15\textwidth]{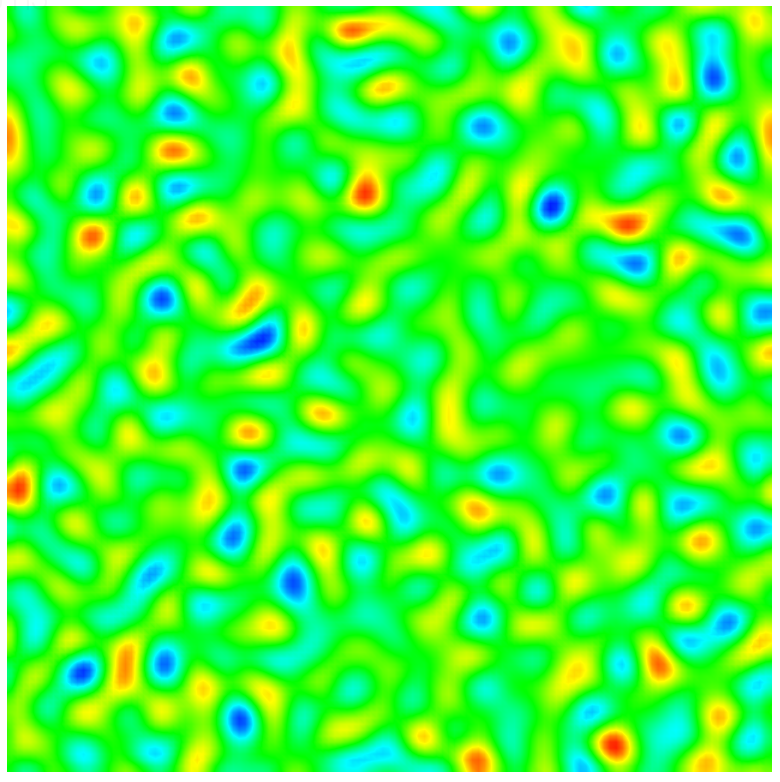}
\includegraphics[width=0.15\textwidth]{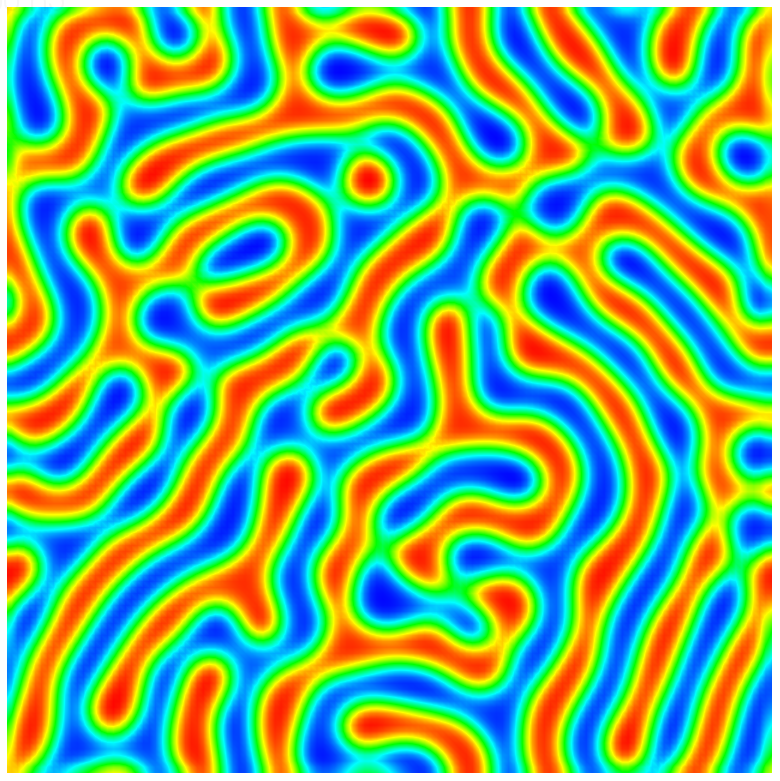}
\includegraphics[width=0.15\textwidth]{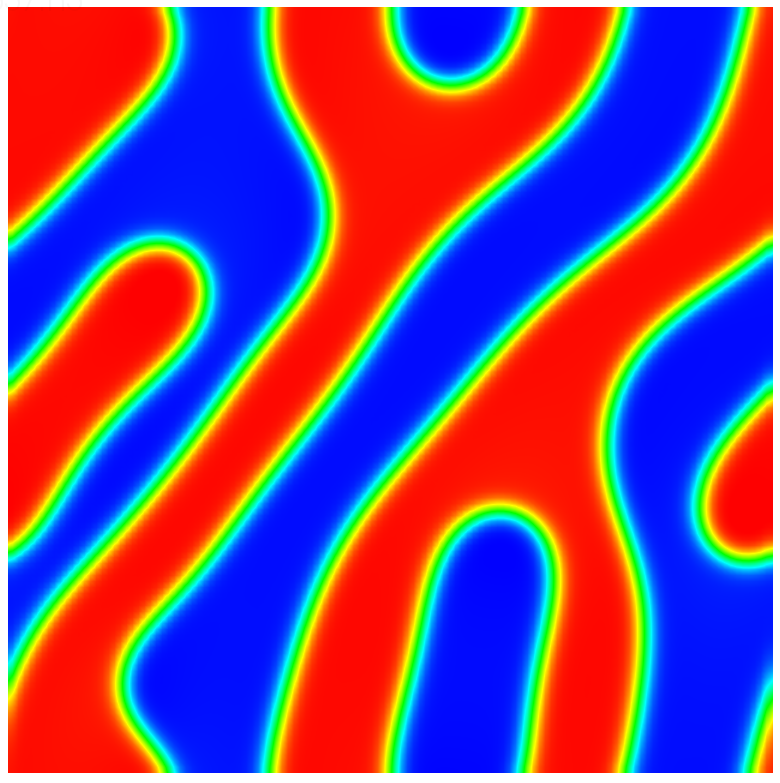}
}

\caption{Comparison of coarsening dynamics with different fractional order $\alpha$. Here we choose $\alpha=0.35,0.5,0.8$ respectively. The profile of $\phi$ at approximately $t=0,0.001,0.2,1$ are plotted.}
\label{fig:Coarsening-phi}
\end{figure}

To further verify this, their corresponding energy evolution with time is shown in Figure \ref{fig:Coarsening-energy}, where we observe the energy with smaller fractional derivative $\alpha$ evolves faster. Also, we observe the energies in all cases (with different time-fractional derivative $\alpha$) decrease with time. Notice in both the continuous theorem \eqref{eq6} and the discrete theorem \eqref{eq37}, it could only be shown that the energy is bounded by the initial energy value. Whether the energy is non increasing in time or not is still an open question.

\begin{figure}[H]
\center
\subfigure[$\alpha=0.35$]{\includegraphics[width=0.3\textwidth]{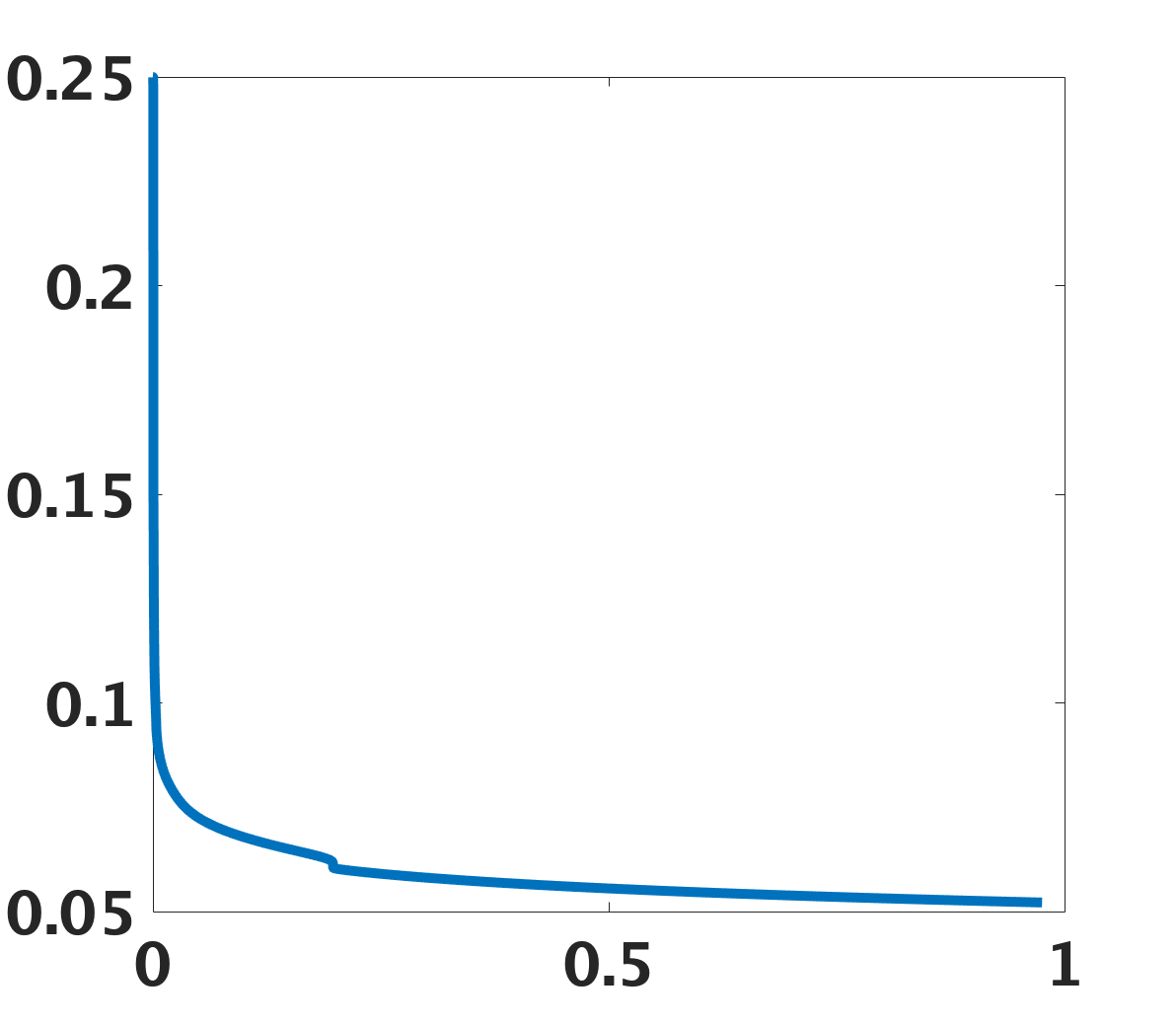}}
\subfigure[$\alpha=0.5$]{\includegraphics[width=0.3\textwidth]{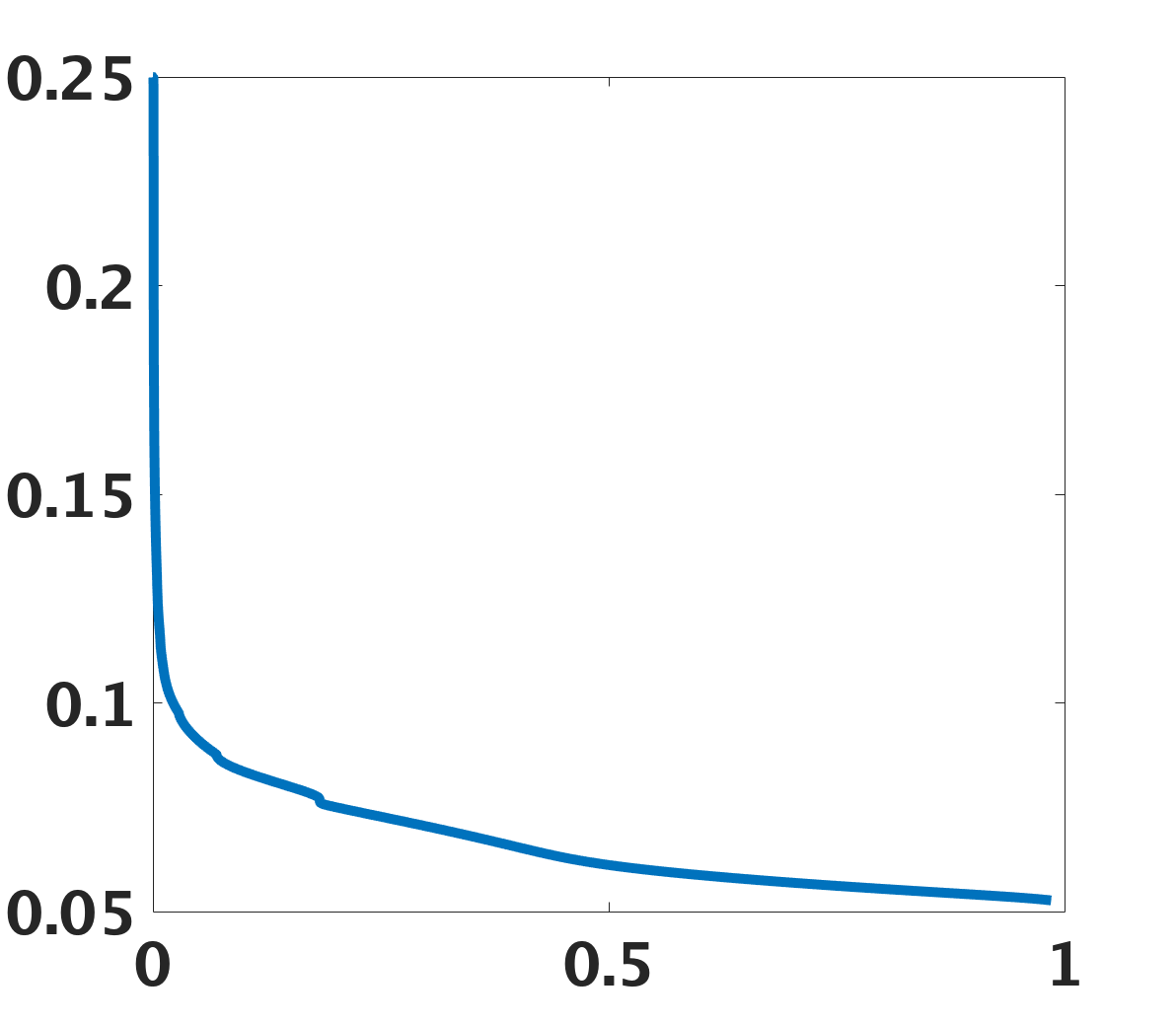}}
\subfigure[$\alpha=0.8$]{\includegraphics[width=0.3\textwidth]{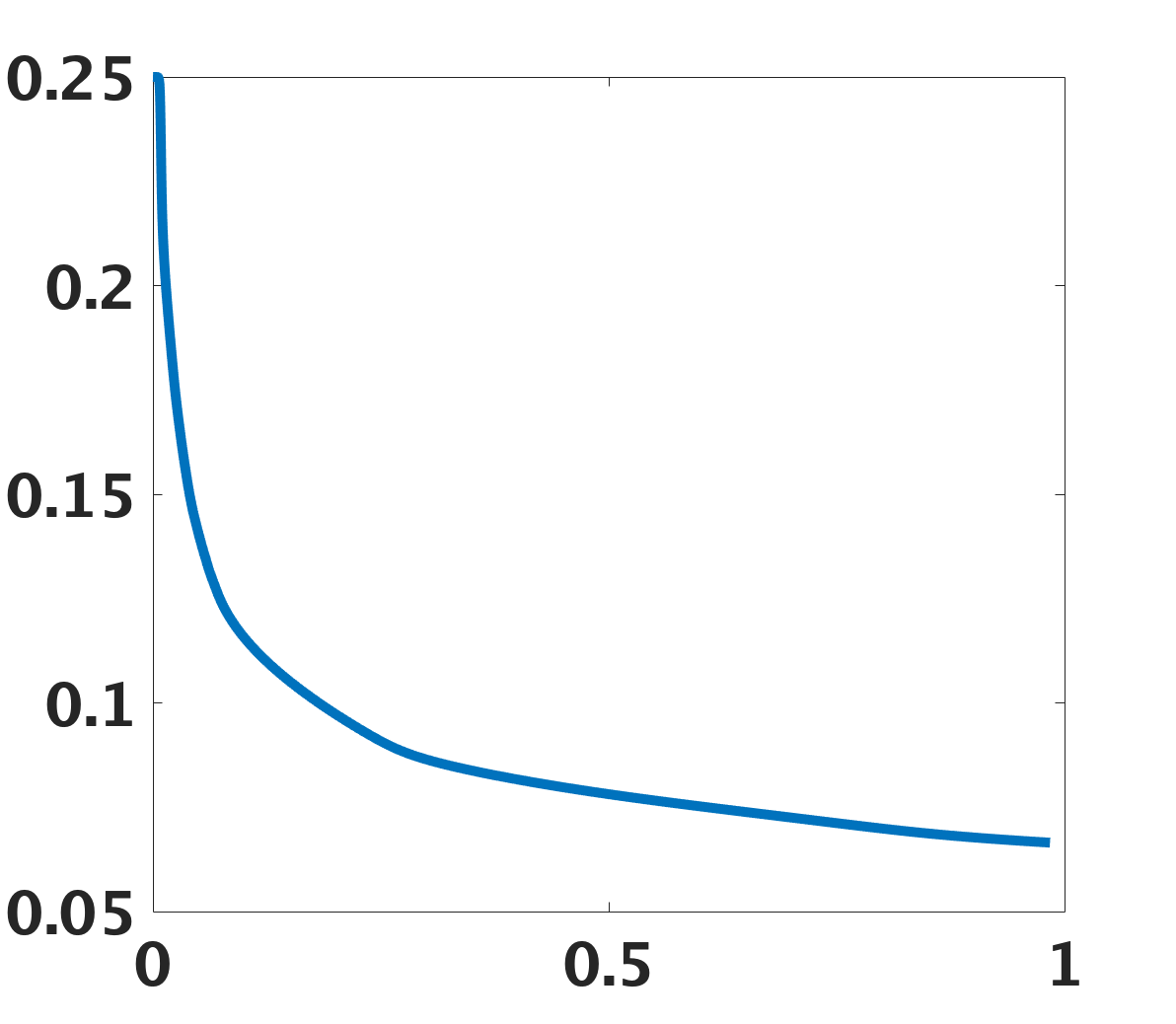}}
\caption{Energy evolution during the coarsening dynamics. This figures show the corresponding energies decreasing with time, with various fractional orders $\alpha=0.35,0.5,0.8$. }
\label{fig:Coarsening-energy}
\end{figure}

In addition, the corresponding time step sizes used for the simulations in Figure \ref{fig:Coarsening-phi} is also summarized in Figure \ref{fig:Coarsening-dt}.
We observe that the time step size increases when the energy evolves slower. This adaptive time strategy has saved the computational resources significantly.

\begin{figure}[H]
\center
\subfigure[$\alpha=0.35$]{\includegraphics[width=0.3\textwidth]{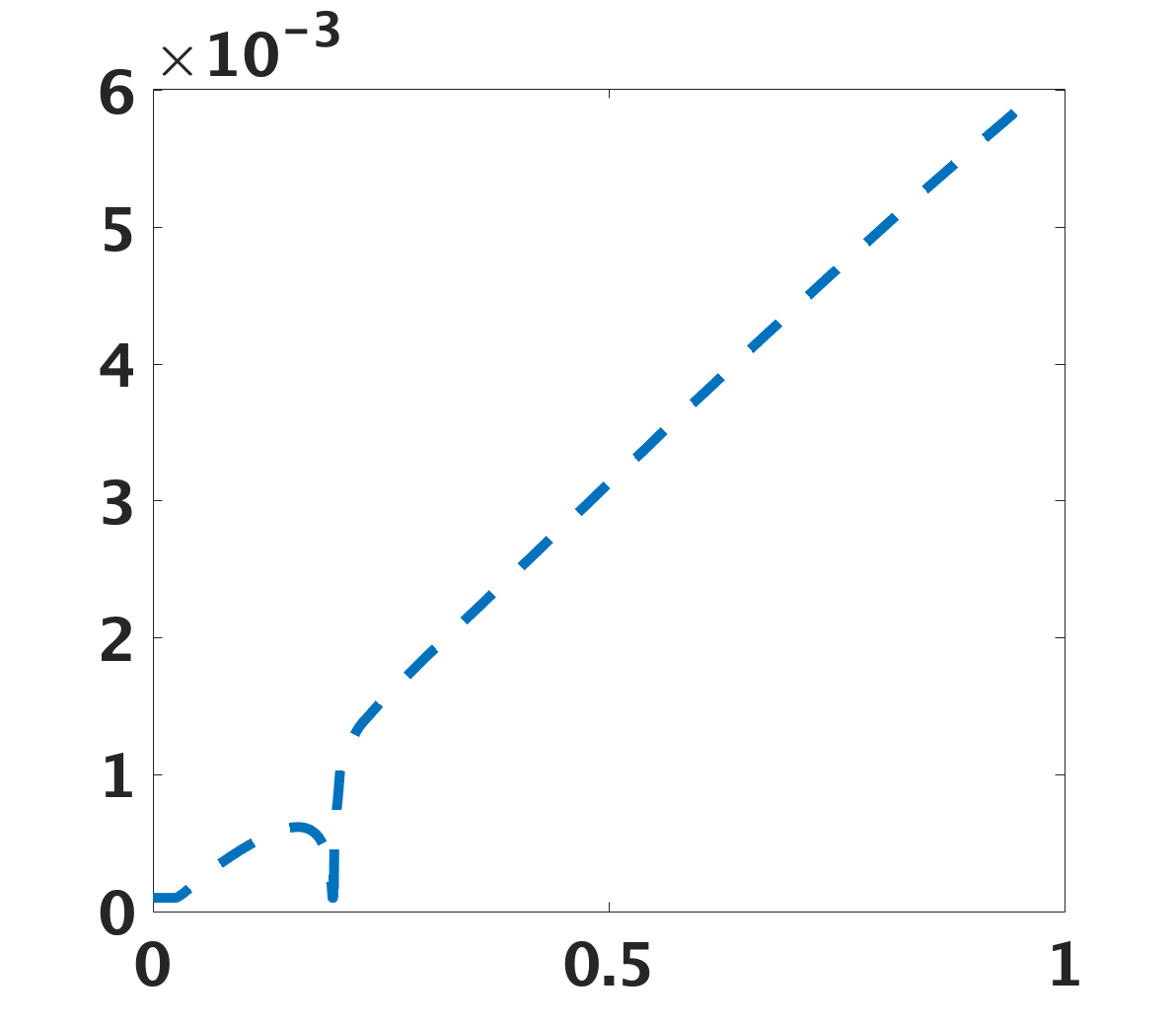}}
\subfigure[$\alpha=0.5$]{\includegraphics[width=0.3\textwidth]{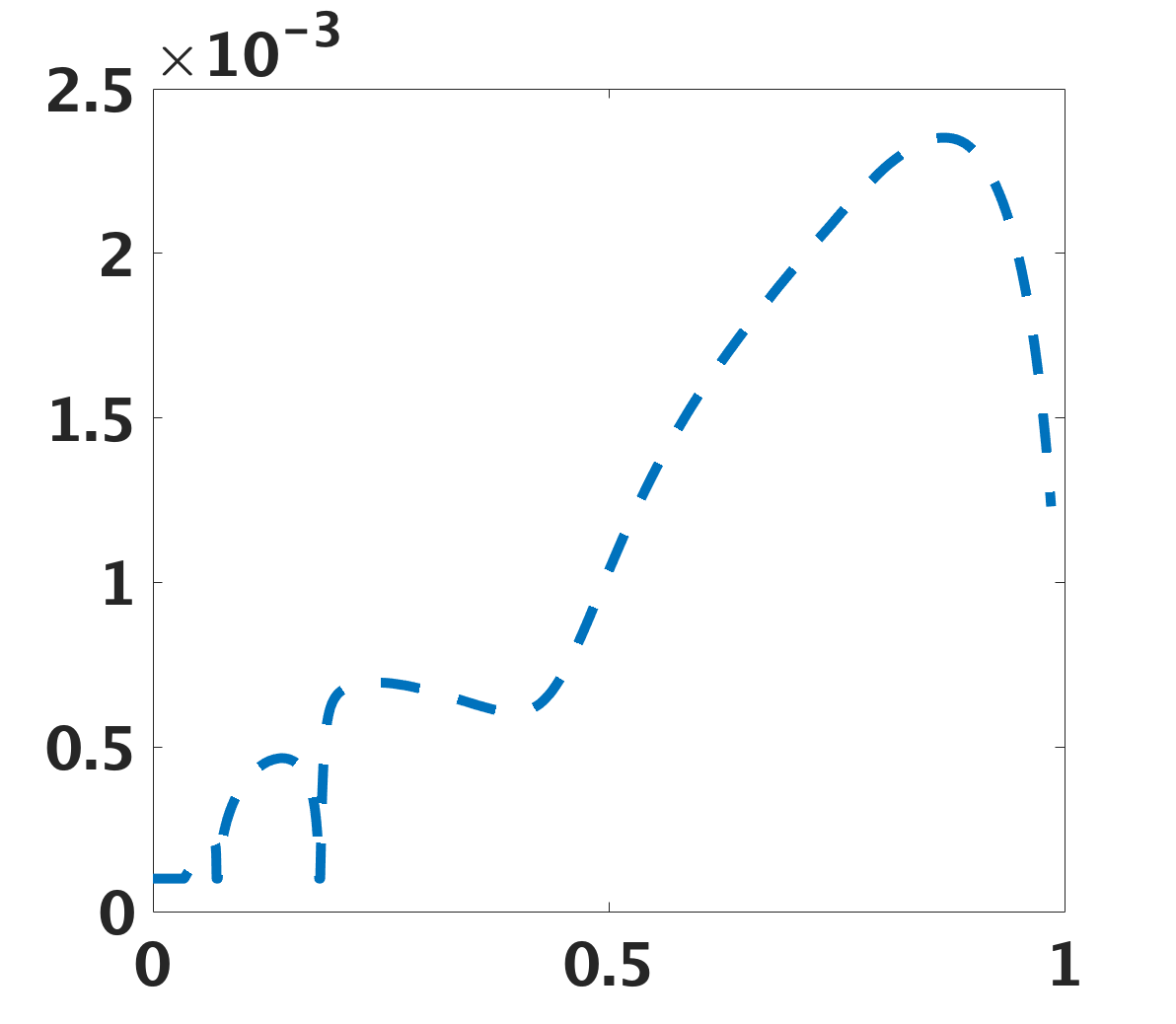}}
\subfigure[$\alpha=0.8$]{\includegraphics[width=0.3\textwidth]{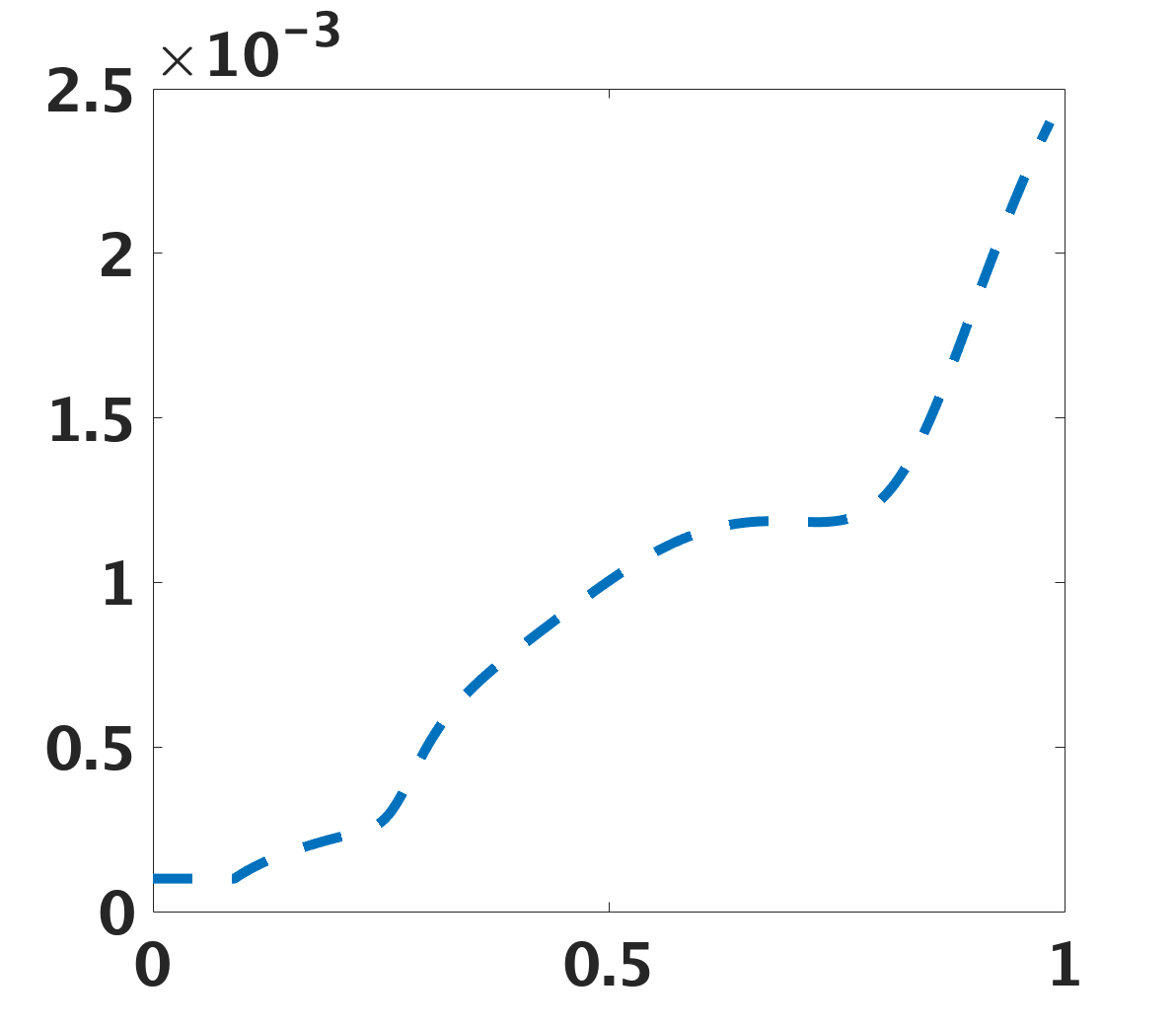}}
\caption{Adaptive time step sizes used during the simulation with various fractional order $\alpha=0.35, 0.5, 0.8$.}
\label{fig:Coarsening-dt}
\end{figure}

\subsection{Coarsening dynamics with various initial profiles}
Next, we investigate how the initial profile would affect the coarsening dynamics. Mainly, we use the same parameters as previous subsection, i.e. $\varepsilon=0.01$, $M = 0.1$. We set the domain  $\Omega=[0 \,\,\, 1] \times [0 \,\,\, 1]$ with $128^2$ meshes, and use the following initial condition
\beq
\phi(x,y, t=0)= \overline{\phi}_0 +  10^{-3} rand(x,y),
\eeq
with $rand(x,y) \in [-1 \,\,\, 1]$ random numbers with uniform distribution, and $\overline{\phi}_0$ a constant. Here we fix the time-fractional derivative $\alpha=0.7$, and use different $\overline{\phi}_0$. The numerical profiles for $\phi$ at different times are summarized in Figure \ref{fig:Coarsening-phi0}. We observe that when the volume ratio of two different components is near 1 (for instance, when $\overline{\phi}_0=0$), the driving mechanism for phase separation is spinodal decomposition; when the volume ration of two different components is away from 1, the driving mechanism for phase separation is nucleation. These phenomena are with a strong agreement with the phenomena for the integer Cahh-Hilliard equation. Also, for the nucleation dynamics, the component with less volume fraction will form droplets.

\begin{figure}[H]
\center
\subfigure[$\overline{\phi}_0 = -0.3$]{
\includegraphics[width=0.15\textwidth]{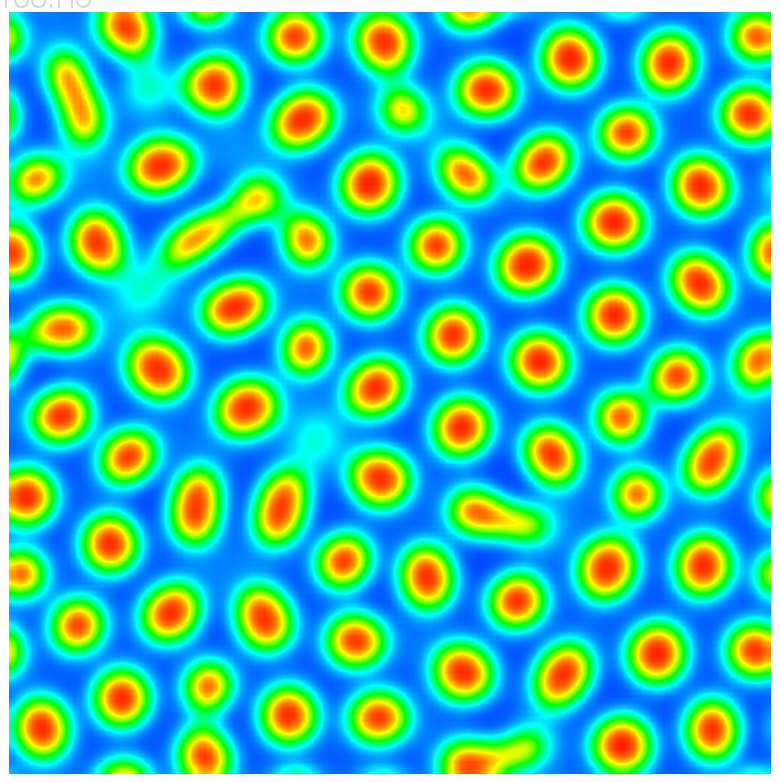}
\includegraphics[width=0.15\textwidth]{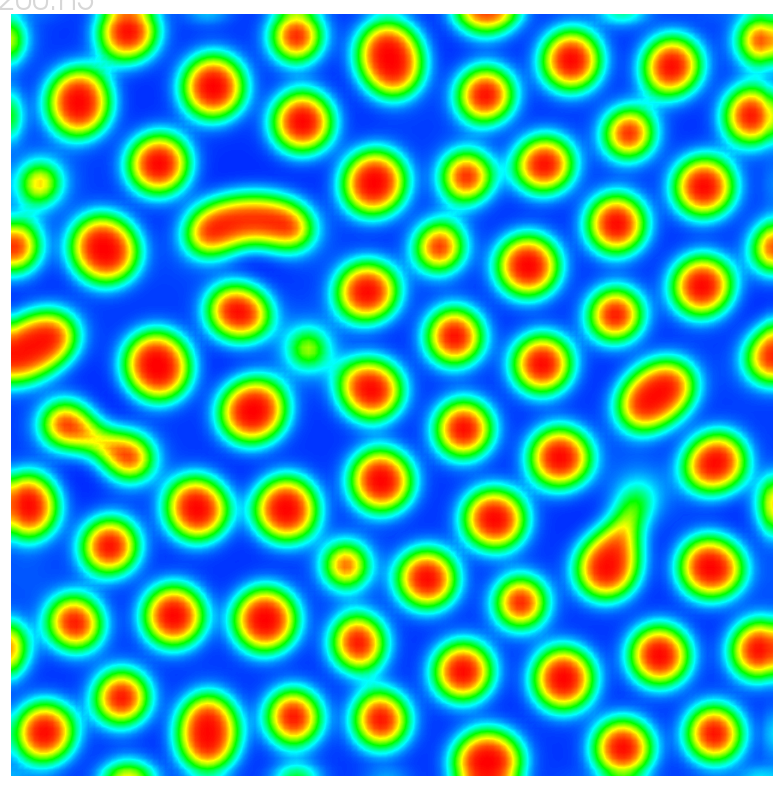}
\includegraphics[width=0.15\textwidth]{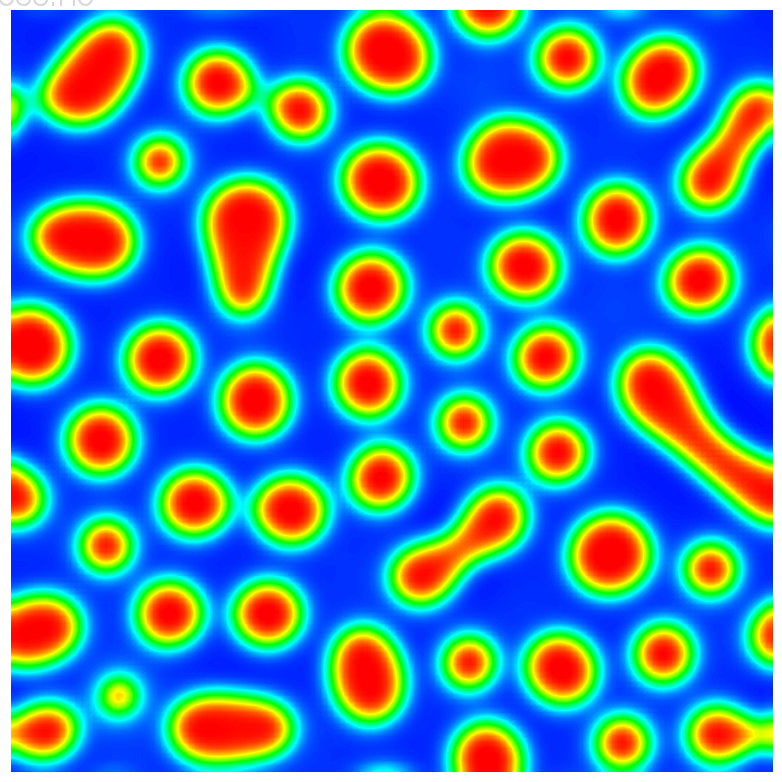}
\includegraphics[width=0.15\textwidth]{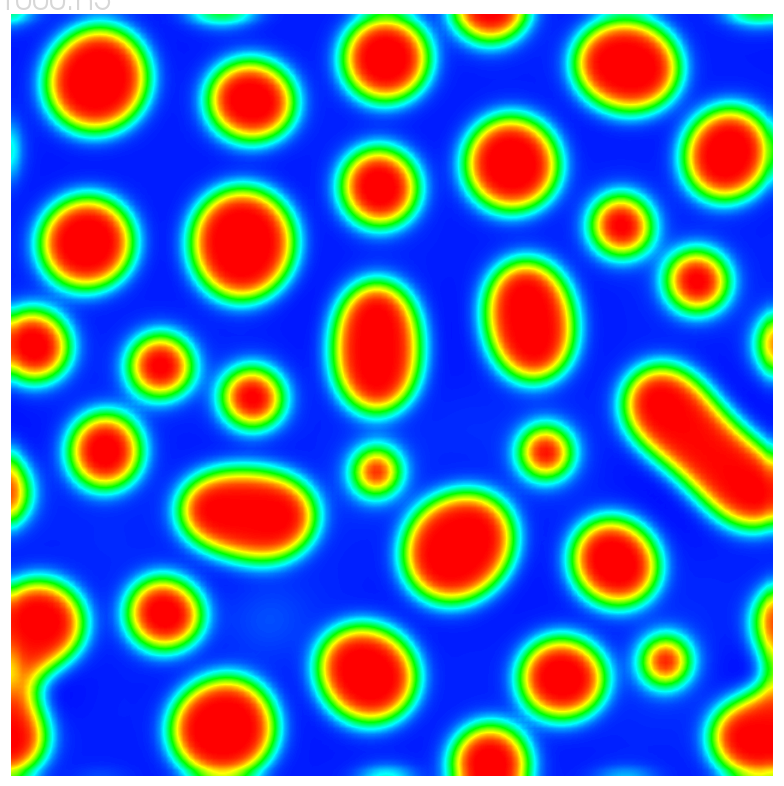}
\includegraphics[width=0.15\textwidth]{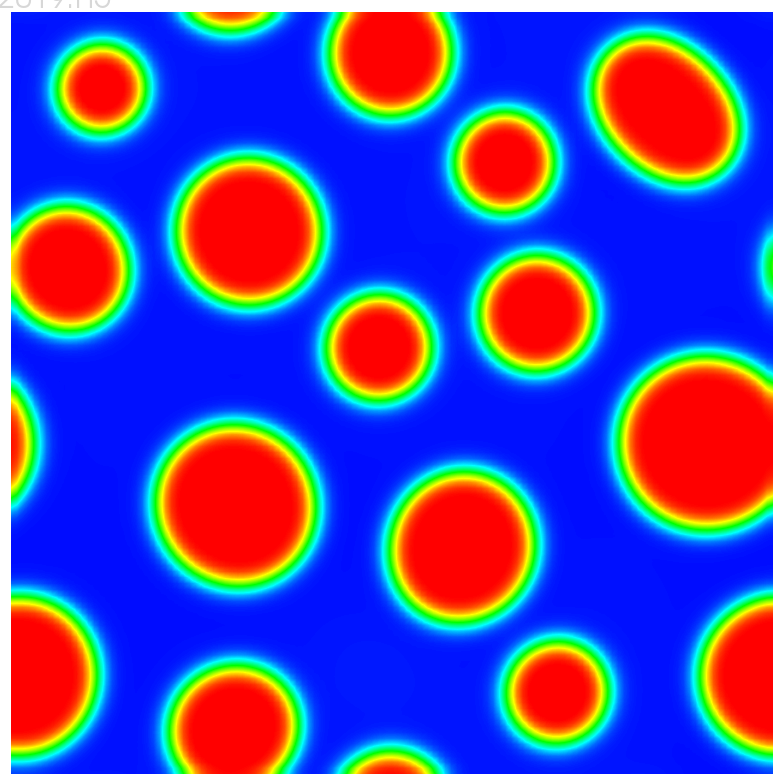}
}

\subfigure[$\overline{\phi}_0=-0.1$]{
\includegraphics[width=0.15\textwidth]{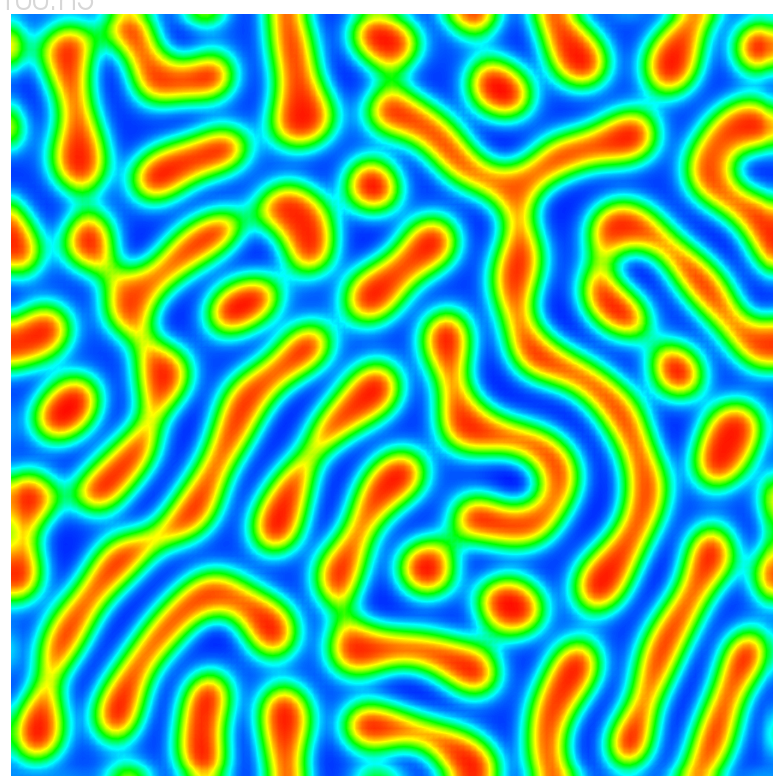}
\includegraphics[width=0.15\textwidth]{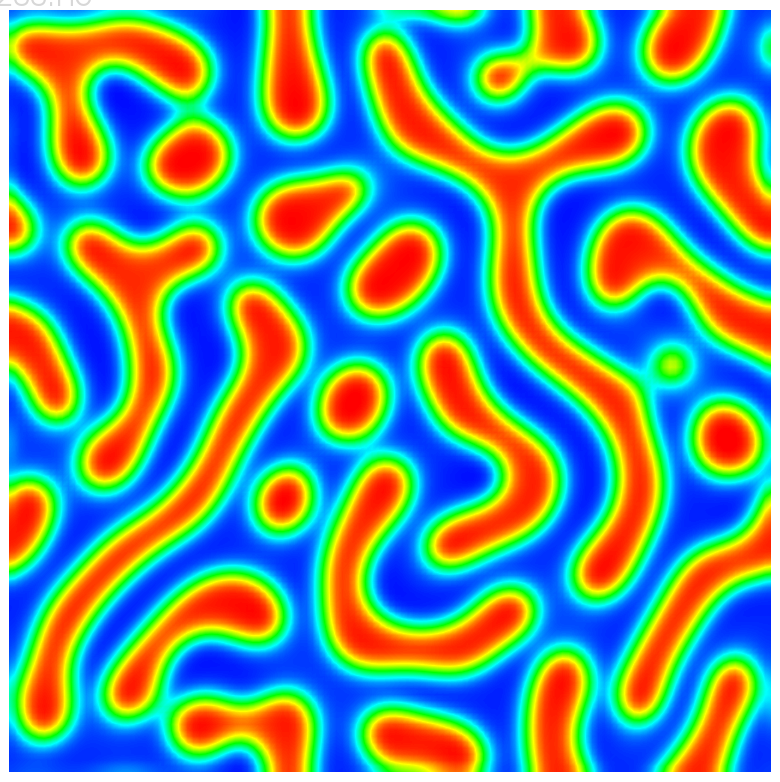}
\includegraphics[width=0.15\textwidth]{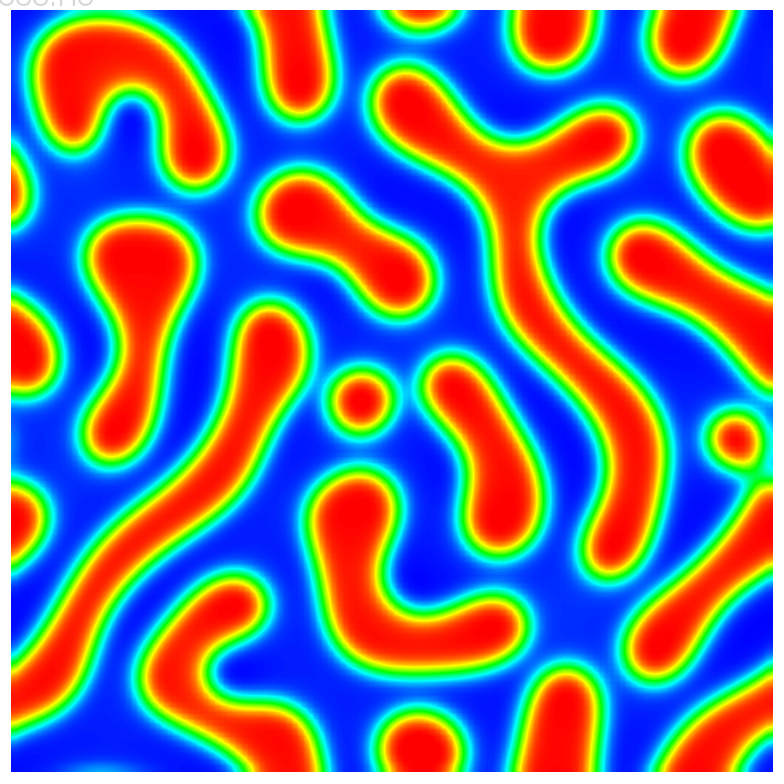}
\includegraphics[width=0.15\textwidth]{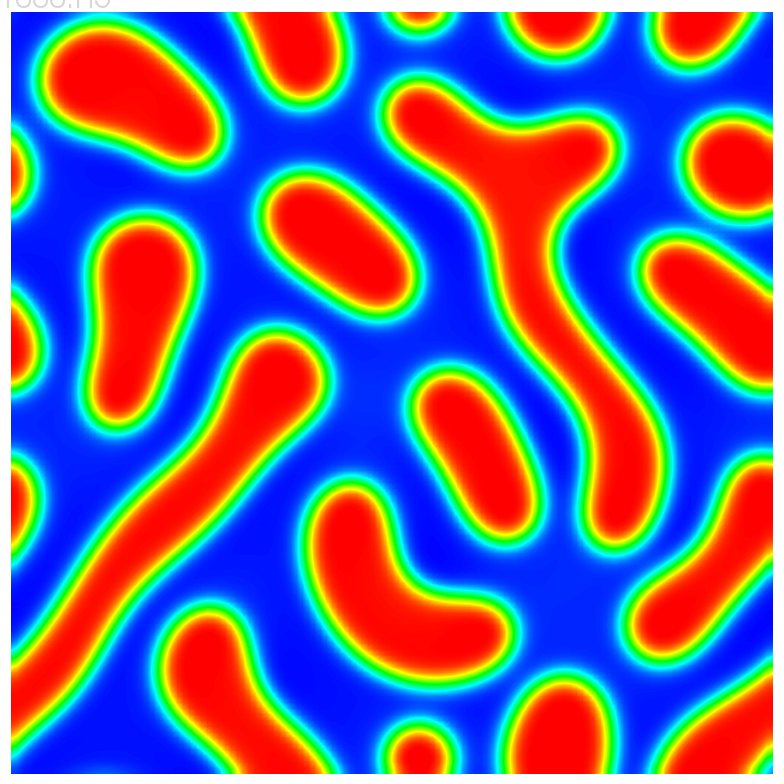}
\includegraphics[width=0.15\textwidth]{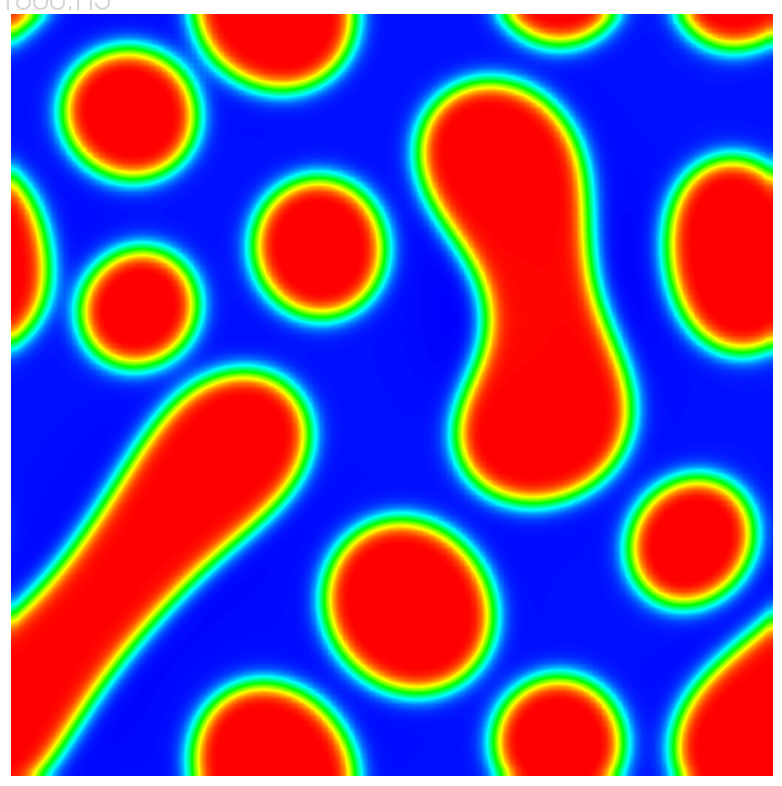}
}

\subfigure[$\overline{\phi}_0=0$]{
\includegraphics[width=0.15\textwidth]{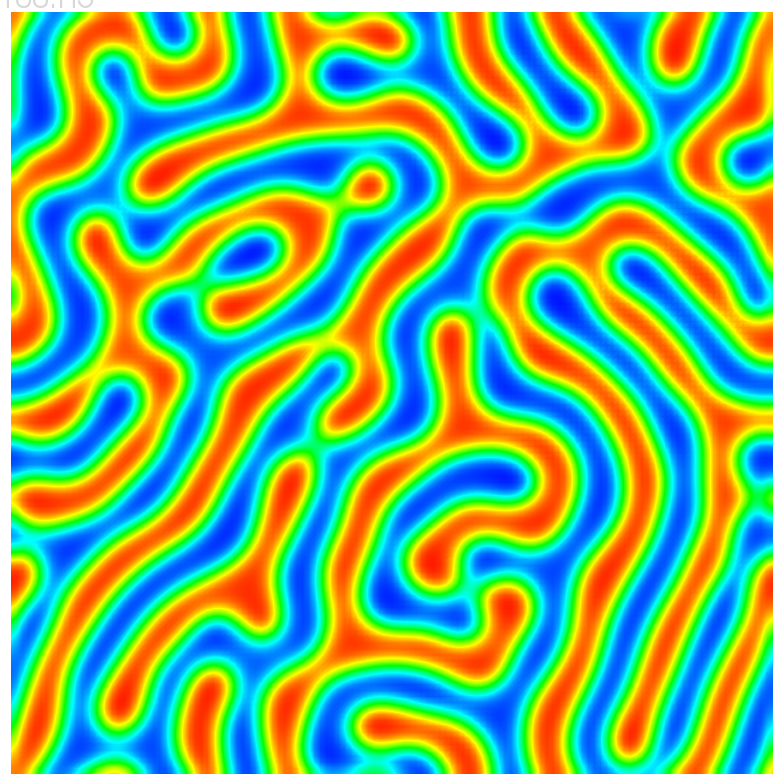}
\includegraphics[width=0.15\textwidth]{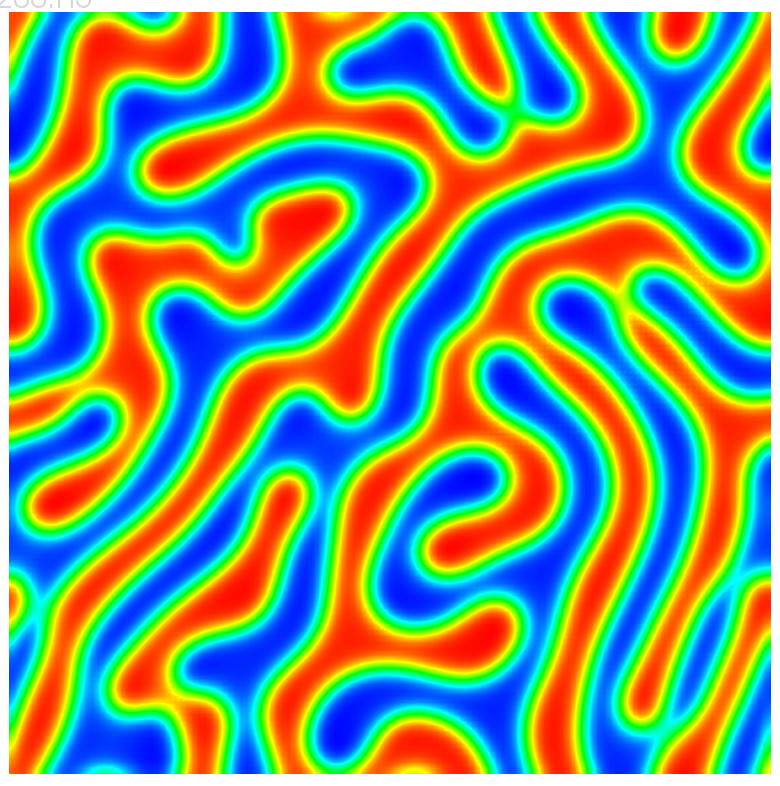}
\includegraphics[width=0.15\textwidth]{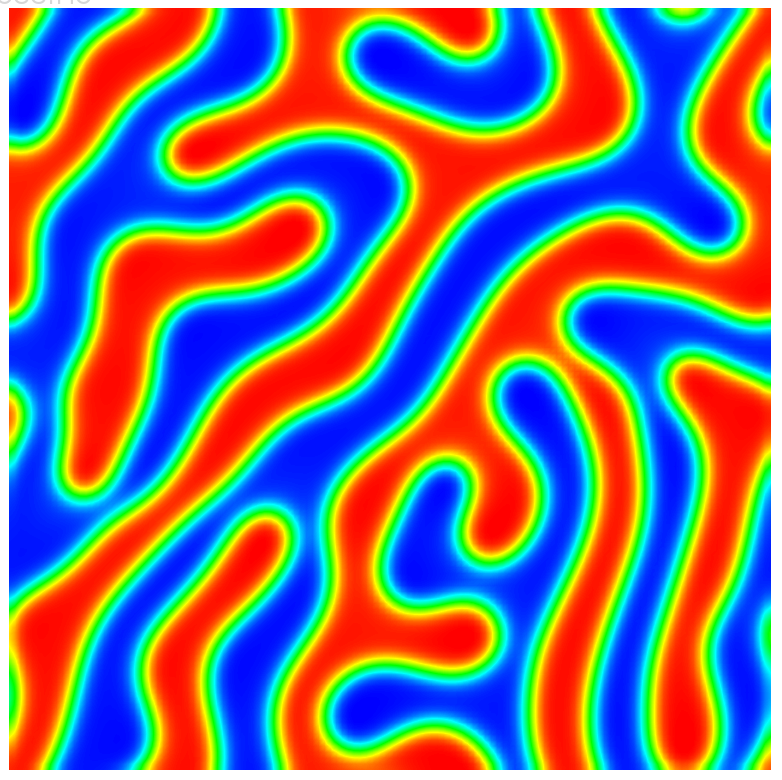}
\includegraphics[width=0.15\textwidth]{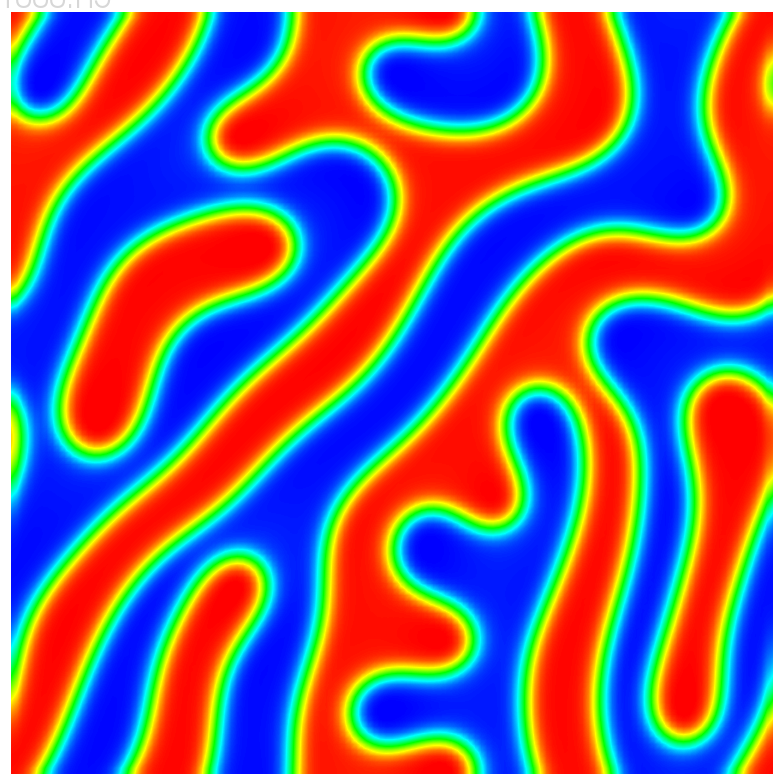}
\includegraphics[width=0.15\textwidth]{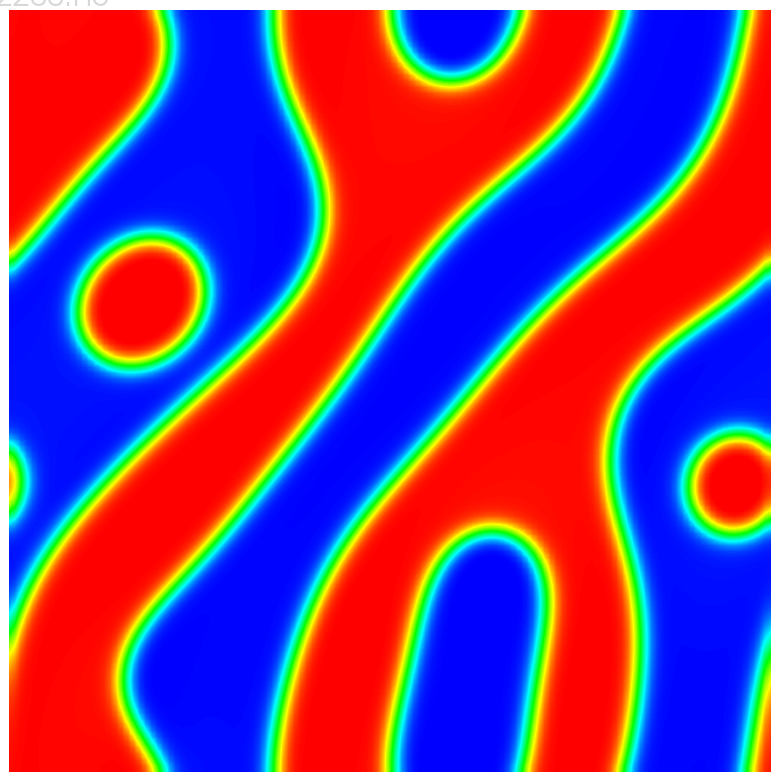}
}

\subfigure[$\overline{\phi}_0=0.1$]{
\includegraphics[width=0.15\textwidth]{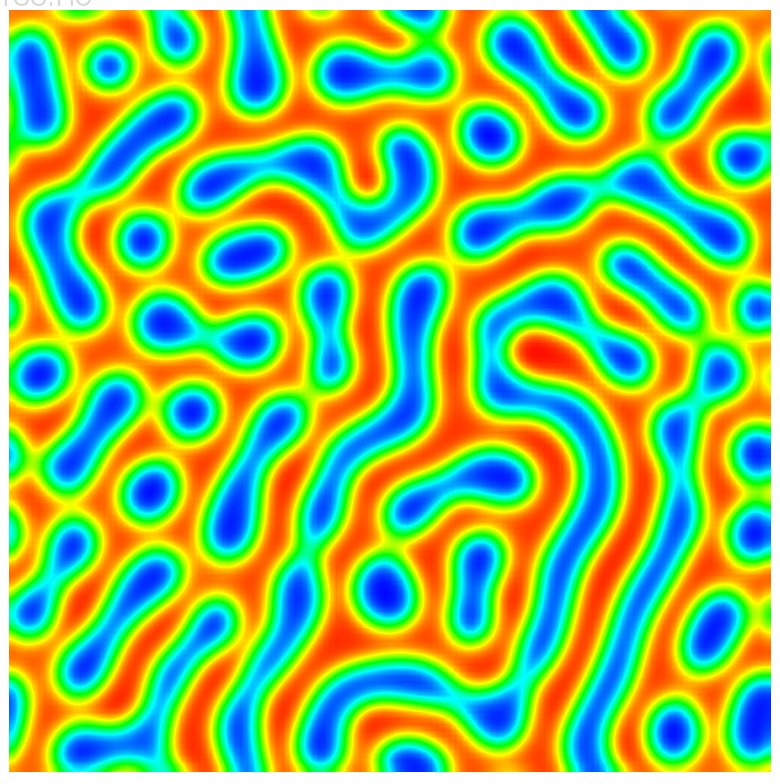}
\includegraphics[width=0.15\textwidth]{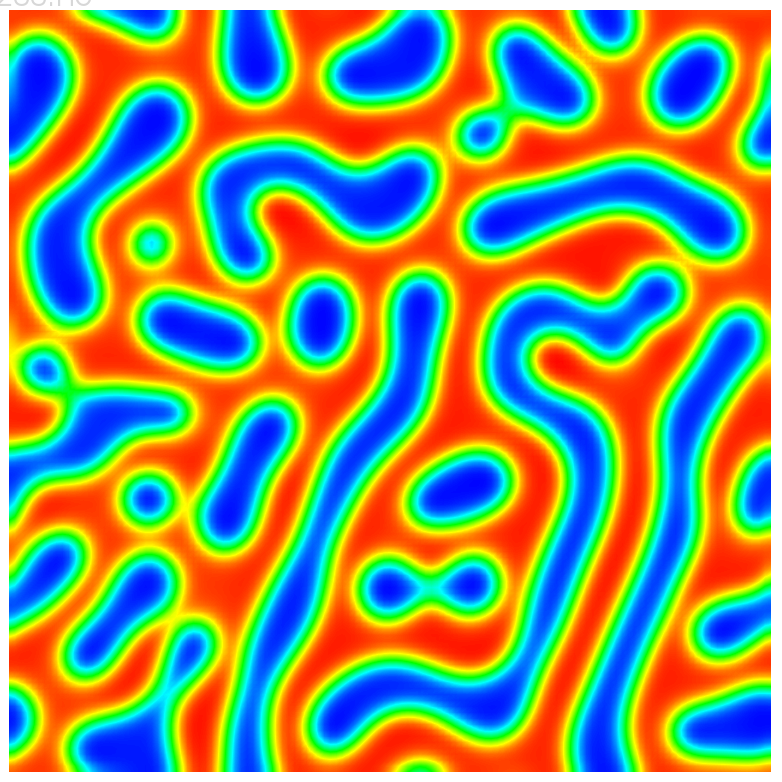}
\includegraphics[width=0.15\textwidth]{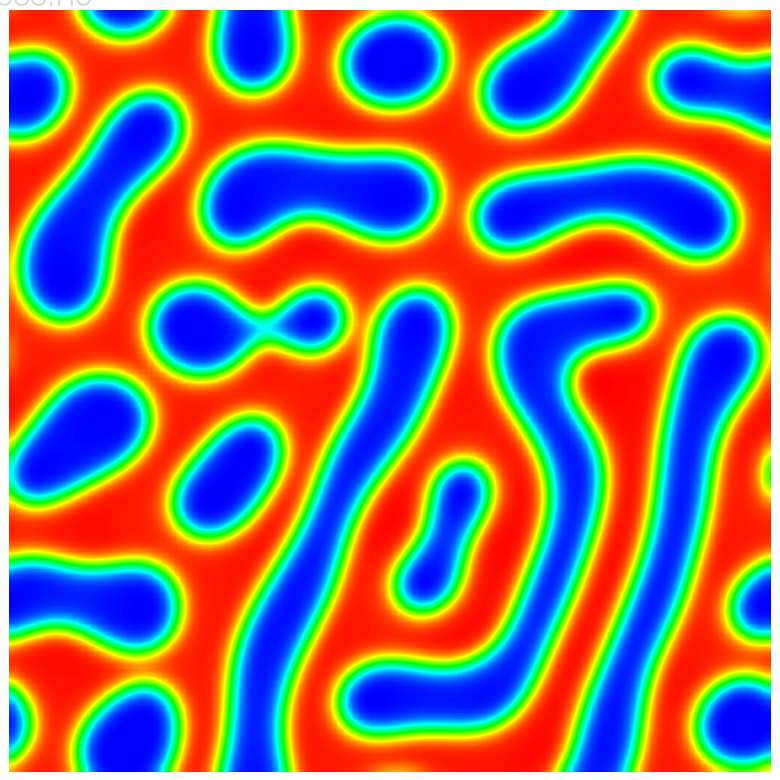}
\includegraphics[width=0.15\textwidth]{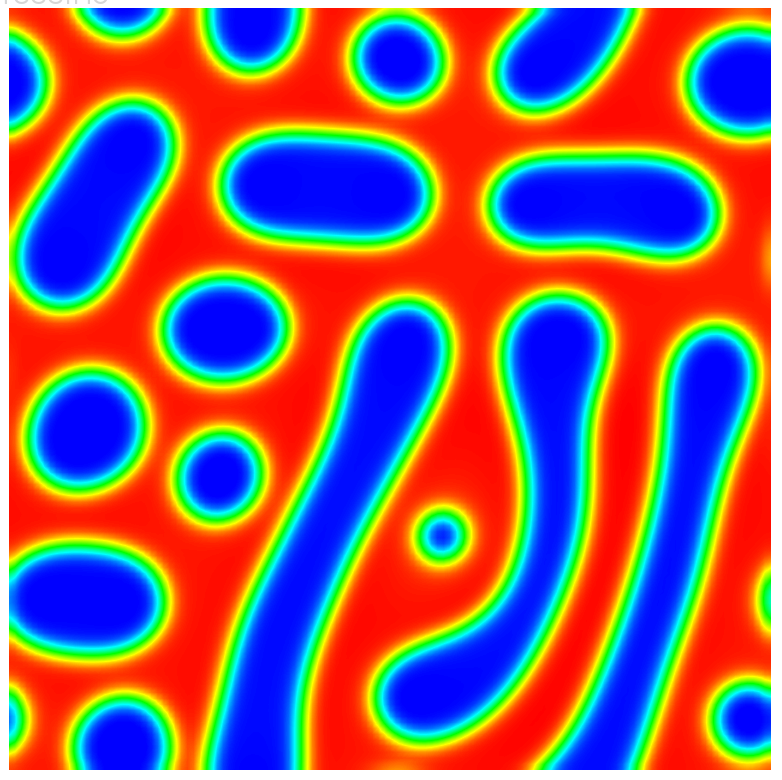}
\includegraphics[width=0.15\textwidth]{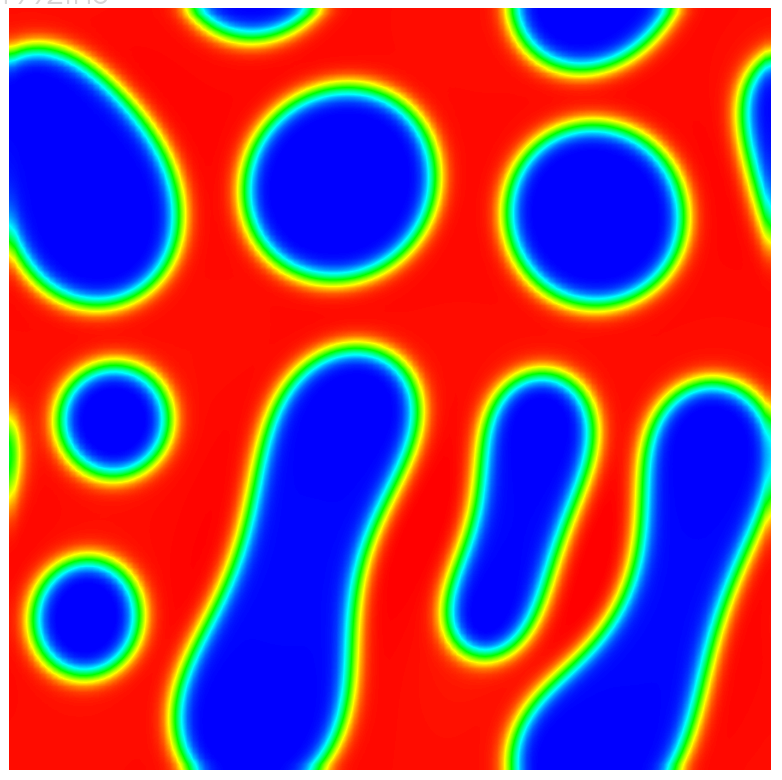}
}

\subfigure[$\overline{\phi}_0=0.3$]{
\includegraphics[width=0.15\textwidth]{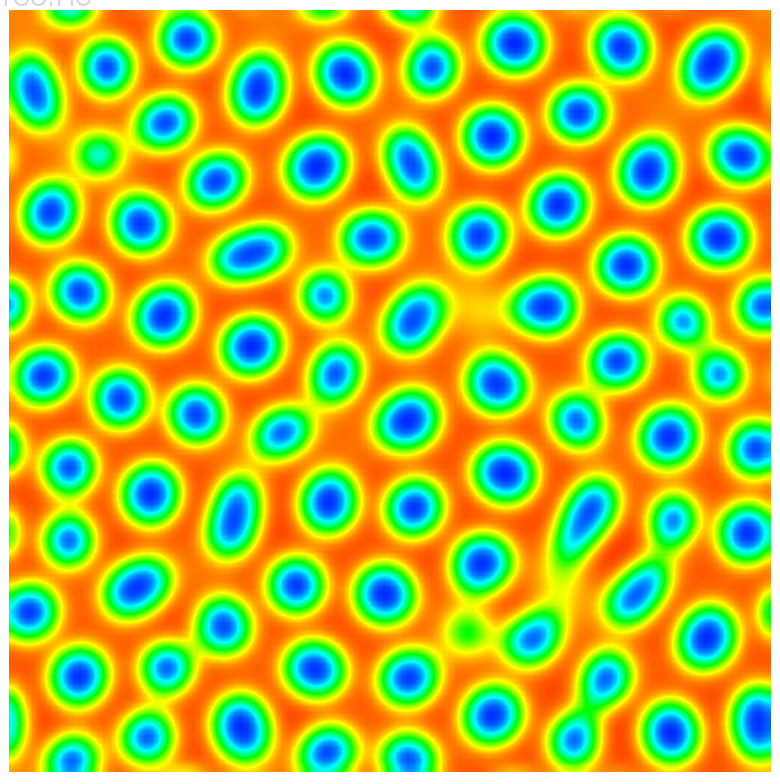}
\includegraphics[width=0.15\textwidth]{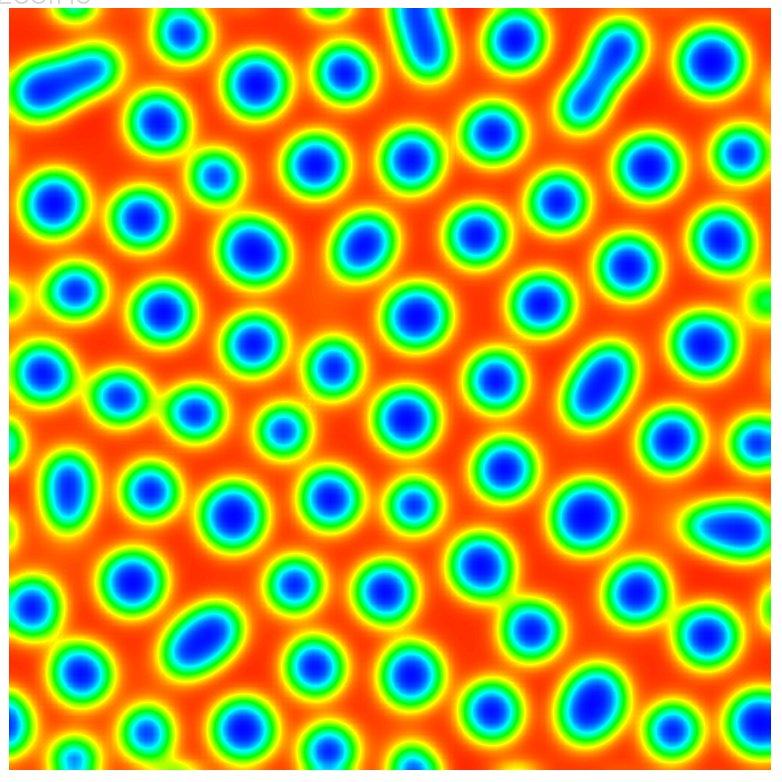}
\includegraphics[width=0.15\textwidth]{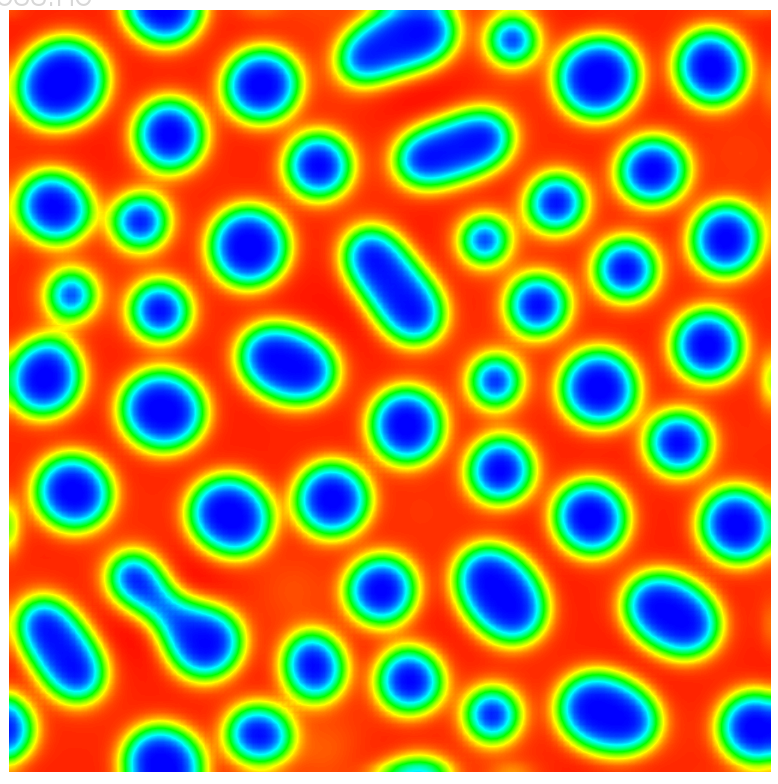}
\includegraphics[width=0.15\textwidth]{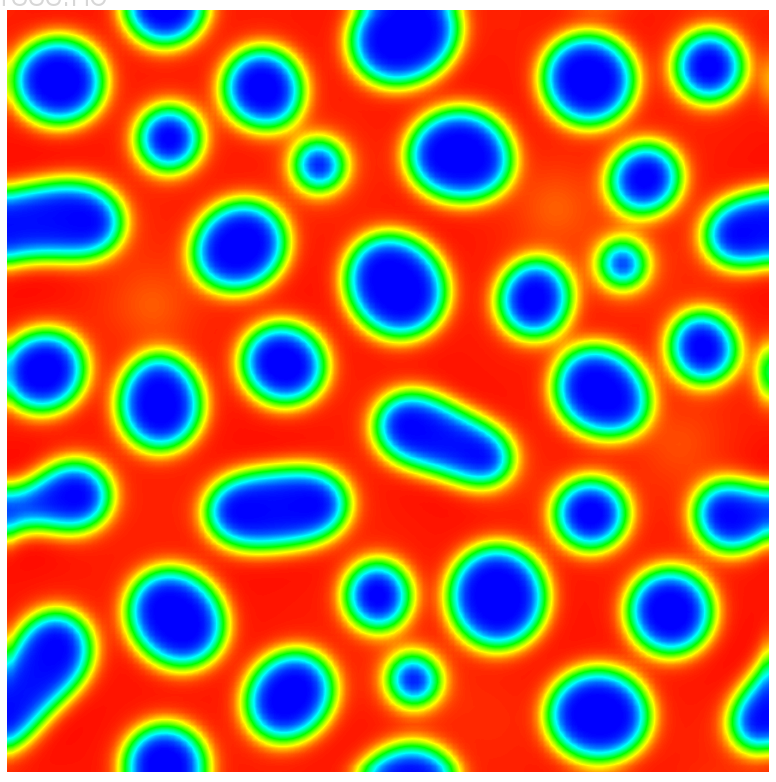}
\includegraphics[width=0.15\textwidth]{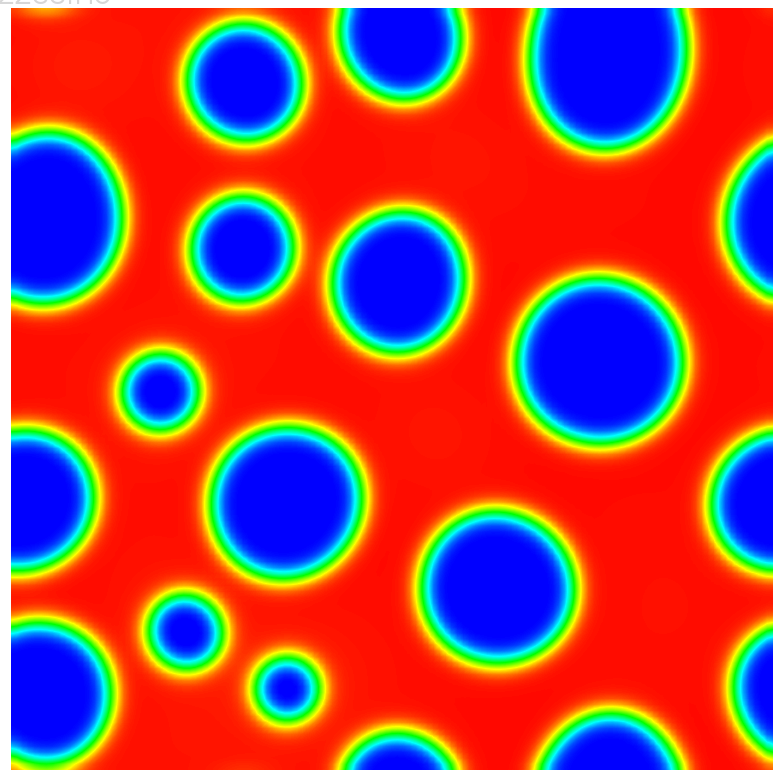}
}

\caption{Comparison of coarsening dynamics with different initial profiles. Here we choose $\overline{\phi}_0=-0.3,-0.1,0,0.1,0.3$ and $\alpha=0.7$ respectively. The profile of $\phi$ at various time slots are plotted. Here red color represents $\phi=1$, and blue color represents $\phi=-1$.}
\label{fig:Coarsening-phi0}
\end{figure}

Next, we study a more complicated case, where the initial averaged concentration varies in space. In specific, we consider the domain $\Omega=[0 \,\,\, 2] \times [0 \,\,\, 1]$, and use the the following initial profile
\beq
\phi(x,y,0) = \frac{1}{2} |x-1| + 10^{-3} rand(x,y),
\eeq
with $rand(x,y)$ generating uniform distributed random number in between -1 and 1. We pick $t_{\max}=10^{-2}$, and $t_{\min}=10^{-5}$. The numerical results with different time fractional $\alpha$ are summarized in Figure \ref{fig:phi1-strip}. We observe that near the middle part of the domain, spinodal decomposition dominates (as the volume fraction ratio is near 1); near both sides of the domain, nucleation dominates, as the one component has more volumes than the other. Also, the one with smaller fractional order $\alpha$ turns to have faster dynamics than those with larger fractional order $\alpha$.

\begin{figure}[H]
\subfigure[$\alpha=0.3$]{
\includegraphics[width=0.24\textwidth]{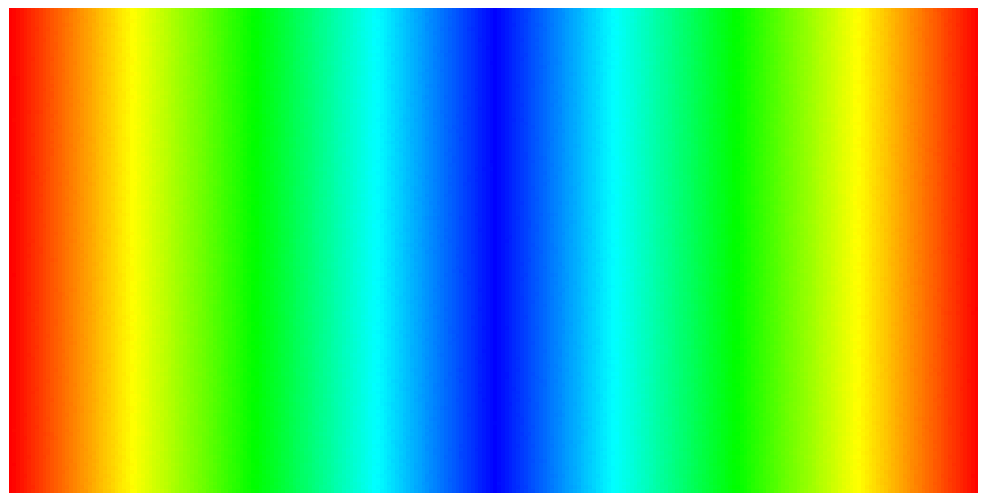}
\includegraphics[width=0.24\textwidth]{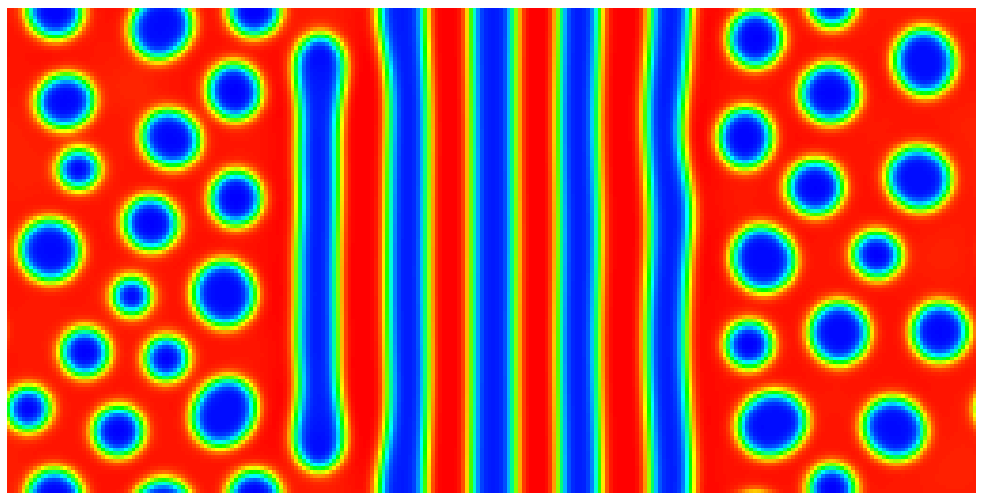}
\includegraphics[width=0.24\textwidth]{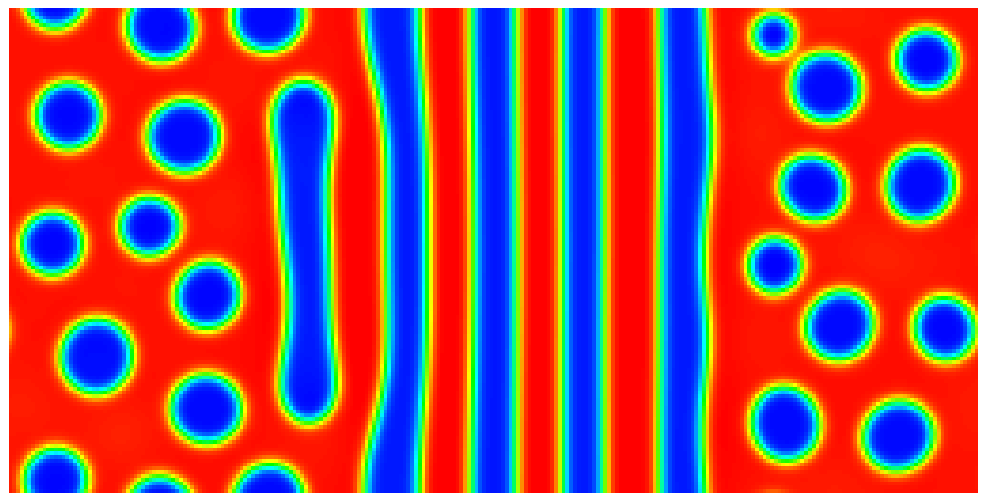}
\includegraphics[width=0.24\textwidth]{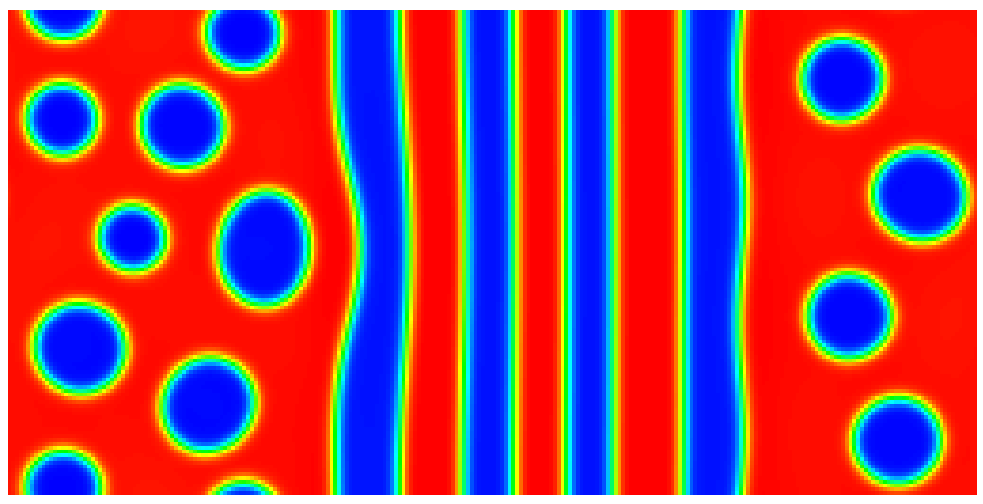}}

\subfigure[$\alpha =0.5$]{
\includegraphics[width=0.24\textwidth]{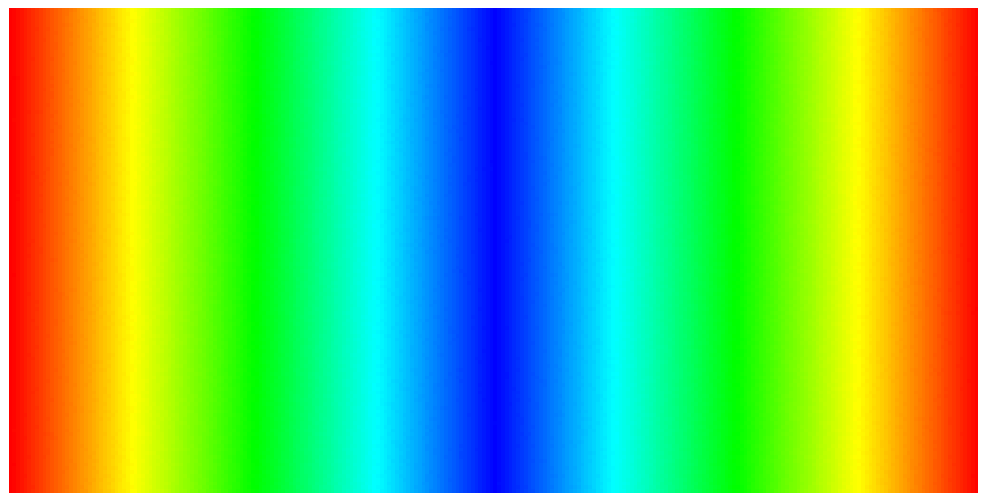}
\includegraphics[width=0.24\textwidth]{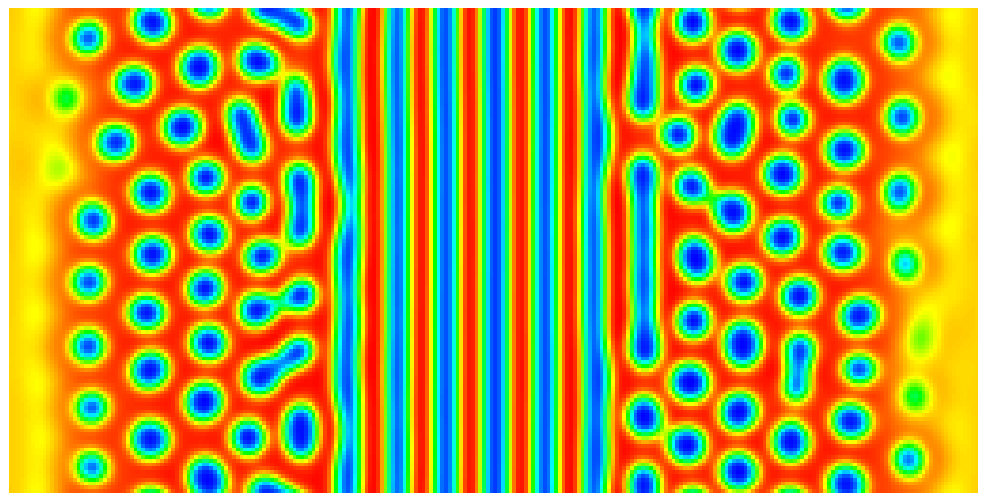}
\includegraphics[width=0.24\textwidth]{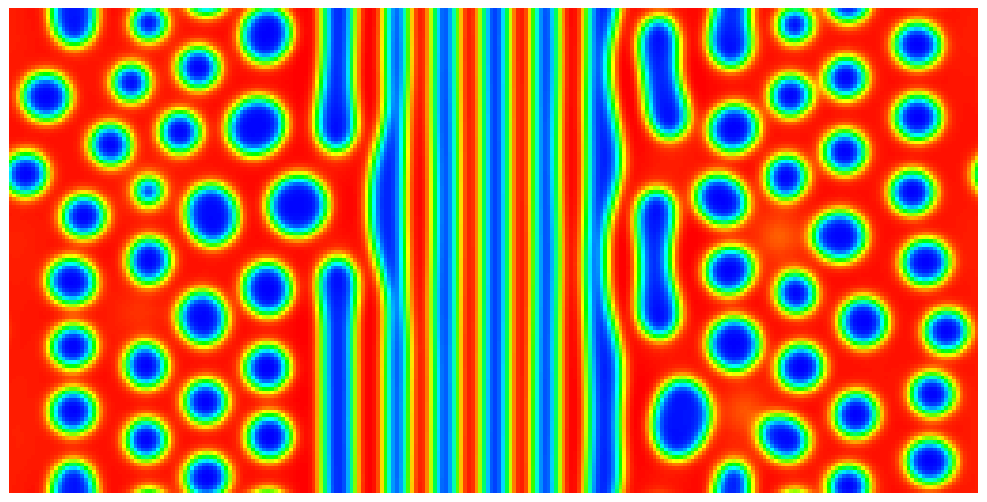}
\includegraphics[width=0.24\textwidth]{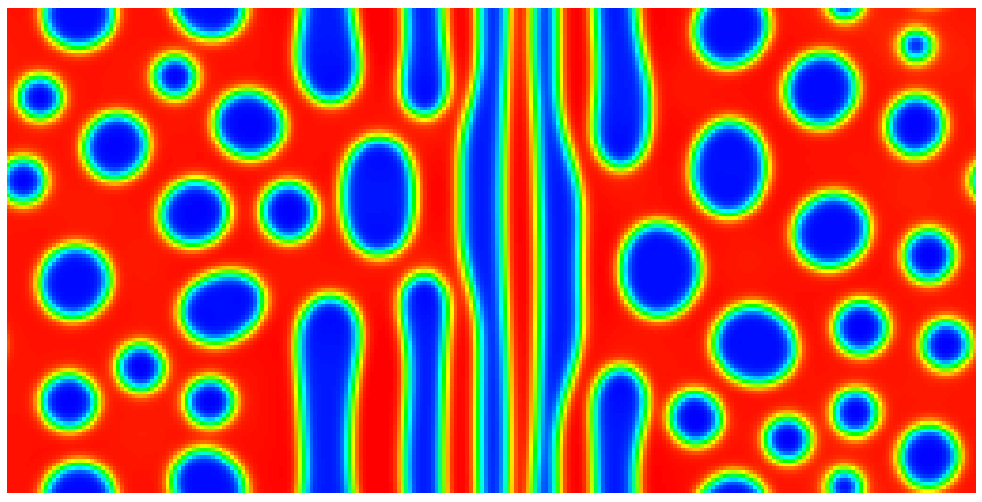}}

\subfigure[$\alpha=0.7$]{
\includegraphics[width=0.24\textwidth]{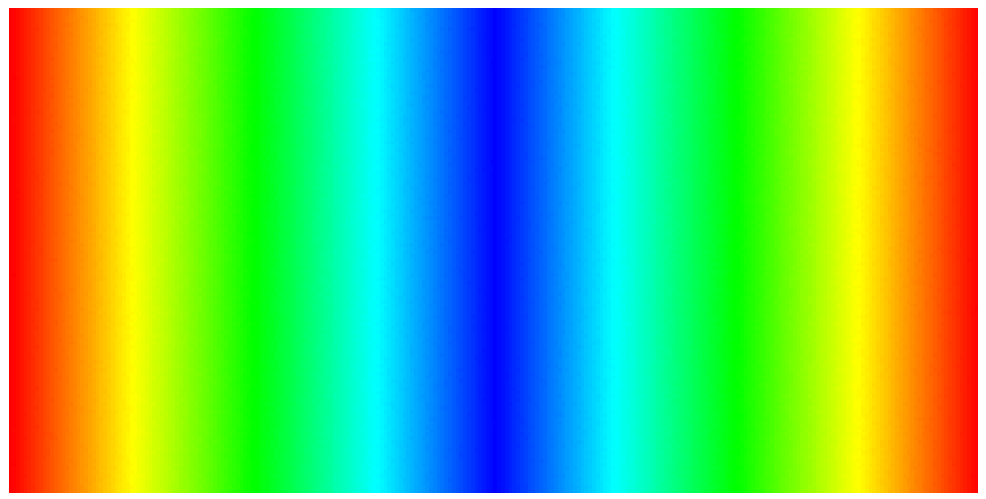}
\includegraphics[width=0.24\textwidth]{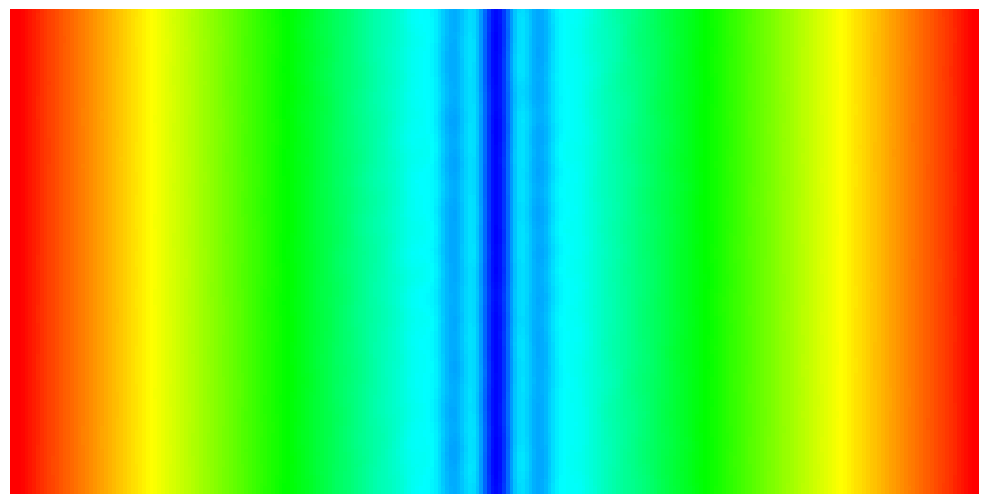}
\includegraphics[width=0.24\textwidth]{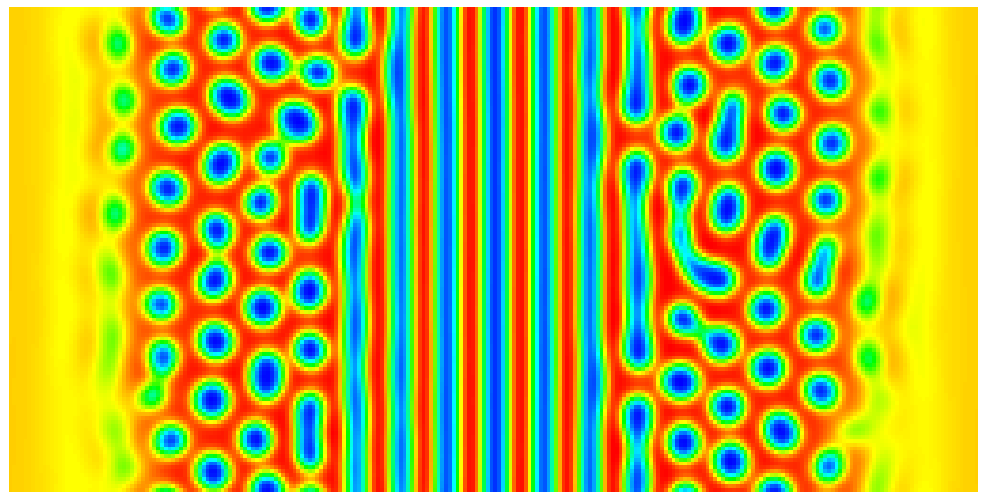}
\includegraphics[width=0.24\textwidth]{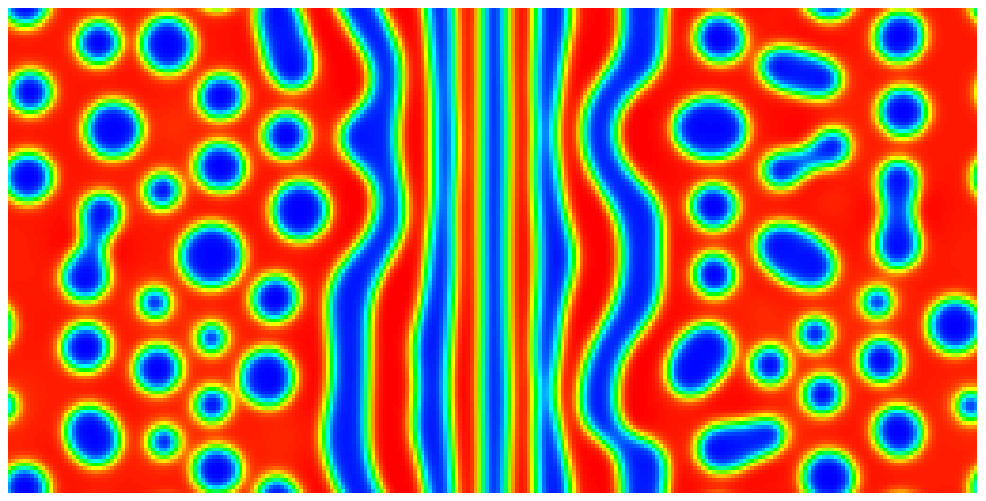}}

\caption{Phase separation dynamics with different fractional order parameter $\alpha$ and an initial condition with spatially dependent average volume fraction ratio. Here the profiles of $\phi$ at different time slots are shown, where (a-c) represents the cases with fractional order $\alpha=0.3, 0.5$ and $0.7$, respectively.}
\label{fig:phi1-strip}
\end{figure}

\subsection{Dynamics of  thin-film rupture during phase separation}
In this subsection, we investigate the dynamics of thin-film rupture during phase separation. We follow a similar setup as in \cite{GomezJCP}. The same model parameters are used as the previous example. Here we set the domain $\Omega = [0 \,\,\, 2] \times [0 \,\,\,\ 1]$, with $256 \times 128$ uniform meshes.  The initial profile for $\phi$ is set as
\beq
\phi(\x, 0 ) = \left\{
\bea{l}
0.001 rand(-1,1), \quad \mbox { if } | x - \frac{5L_x}{6} + \frac{r_0}{2} \sin (10 \pi y)|< r_0, \\
-0.1 +  0.001 rand(-1,1), \quad \mbox { if } | x - \frac{3L_x}{6} + \frac{r_0}{2} \sin (10 \pi y)|< r_0, \\
-0.2 +   0.001 rand(-1,1), \quad \mbox { if } | x - \frac{L_x}{6} + \frac{r_0}{2} \sin (10 \pi y)|< r_0, \\
-1, \quad  \mbox{other wise.} \\
\eea
\right.
\eeq
where $r_0 = 0.05$. The numerical results for $\phi$ at various times are shown in Figure \ref{fig:strip}. We observe that different initial concentrations of liquid thin film show different rupture dynamics, where a spinodal-decomposition driven rupture pattern is observed on the right side, and a nucleation driven rupture pattern is observed on the left side.

\begin{figure}[H]
\center
\subfigure[]{\includegraphics[width=0.3\textwidth]{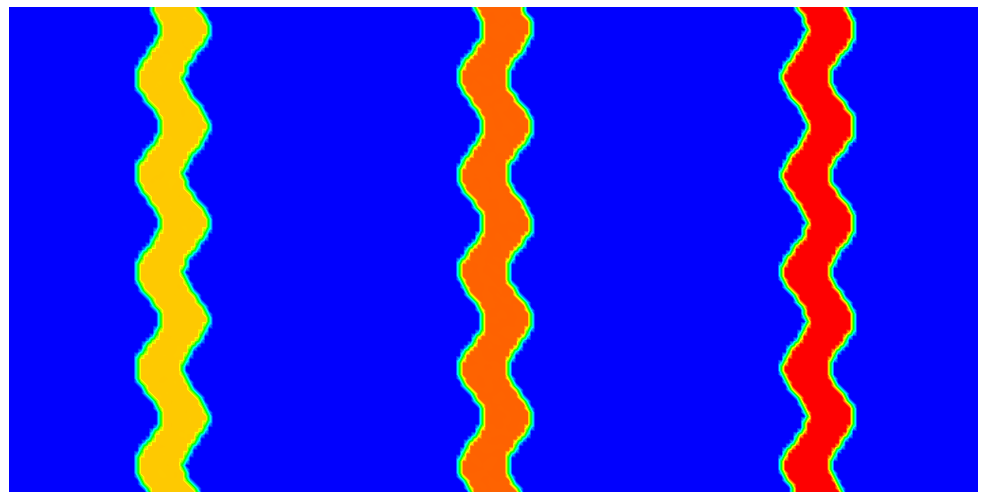}}
\subfigure[]{\includegraphics[width=0.3\textwidth]{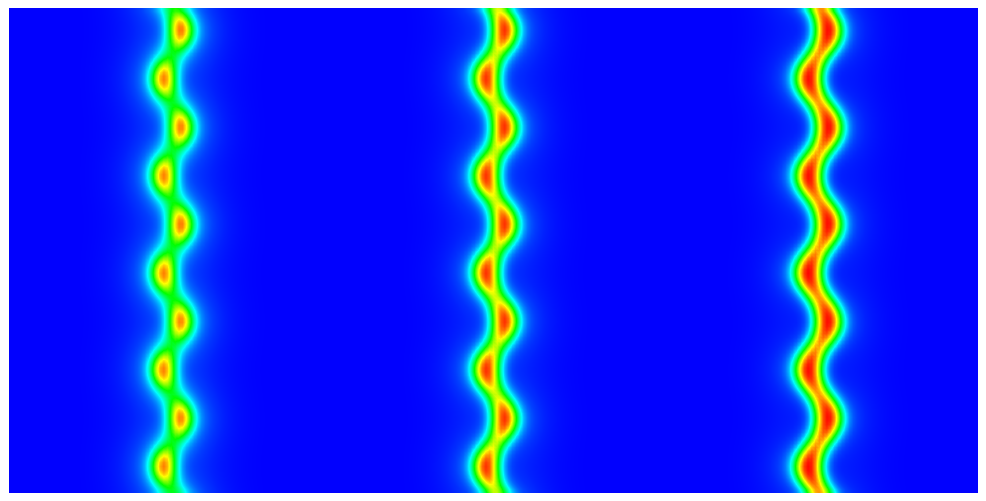}}
\subfigure[]{\includegraphics[width=0.3\textwidth]{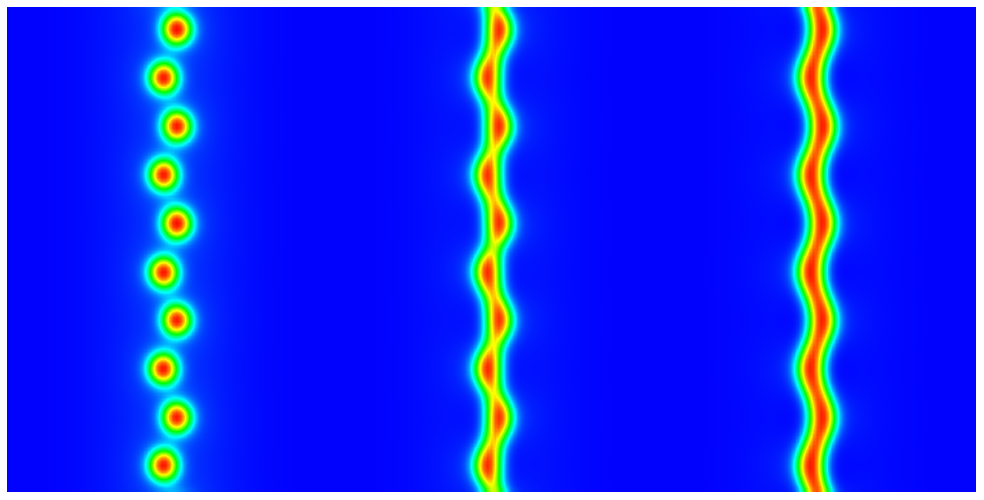}}

\subfigure[]{\includegraphics[width=0.3\textwidth]{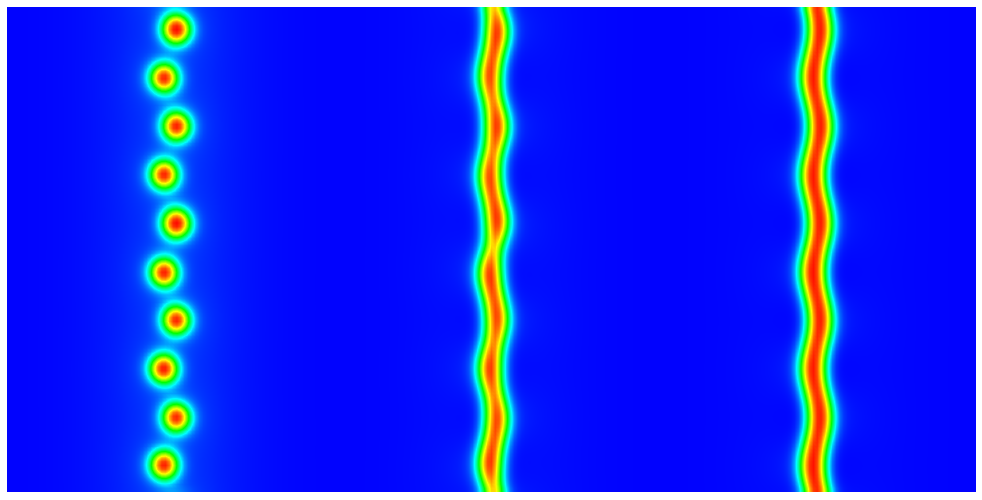}}
\subfigure[]{\includegraphics[width=0.3\textwidth]{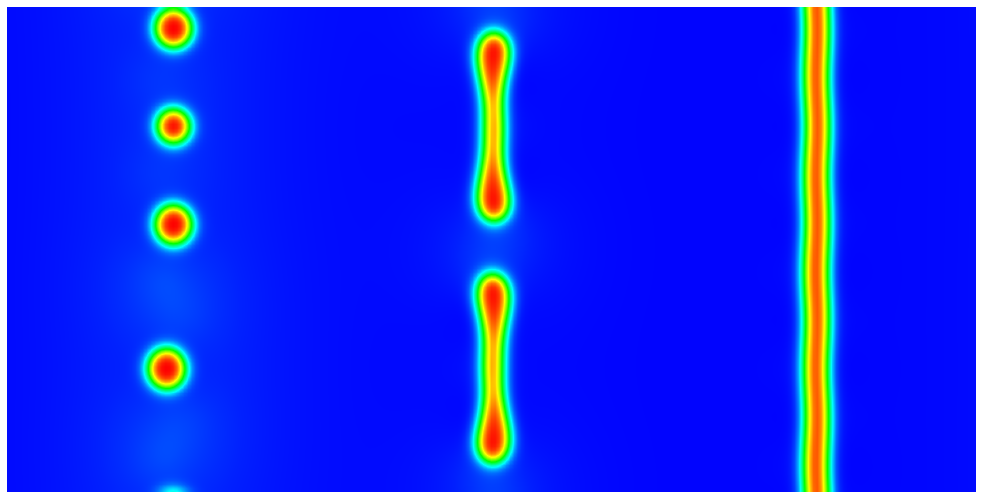}}
\subfigure[]{\includegraphics[width=0.3\textwidth]{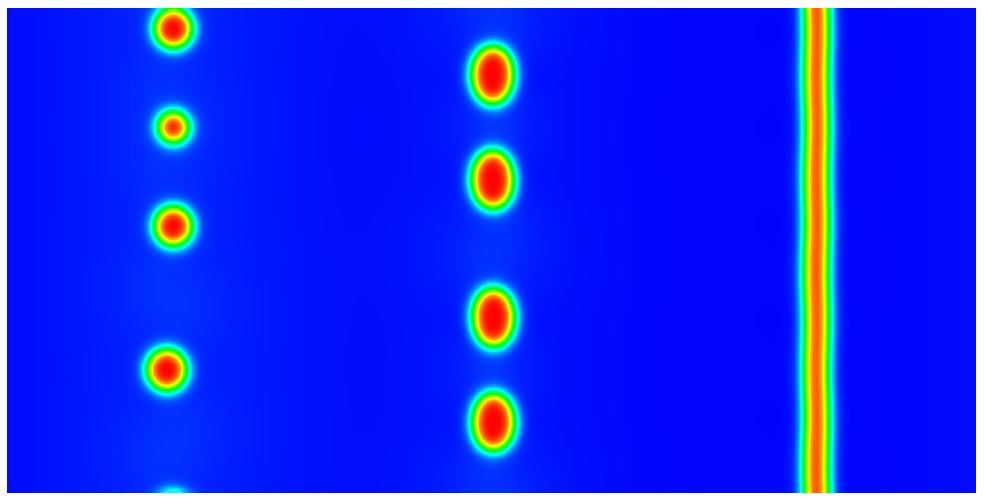}}

\caption{Rupture dynamics of liquid thin-film during phase separation of a binary fluid mixture. Here three thin-film columns (fluid mixture with different volume fraction ratio) are placed in a pure component. Different dynamics are observed due to the volume fraction ratio (nucleation dynamics on the left side, and spinodal dynamics on the right side). The profiles of $\phi$ at various time slots are shown.}
\label{fig:strip}
\end{figure}

In addition, the energy evolution with time is summarized in Figure \ref{fig:strip-energy}(a), where we observe the energy is decreasing in time. Also, the time step size at different times is summarized in Figure \ref{fig:strip-energy}(b), where we observe a relatively large time step is used at the stages when energy is decreasing slowly. This saves the computational time noticeably.

\begin{figure}[H]
\center
\subfigure[]{\includegraphics[width=0.4\textwidth]{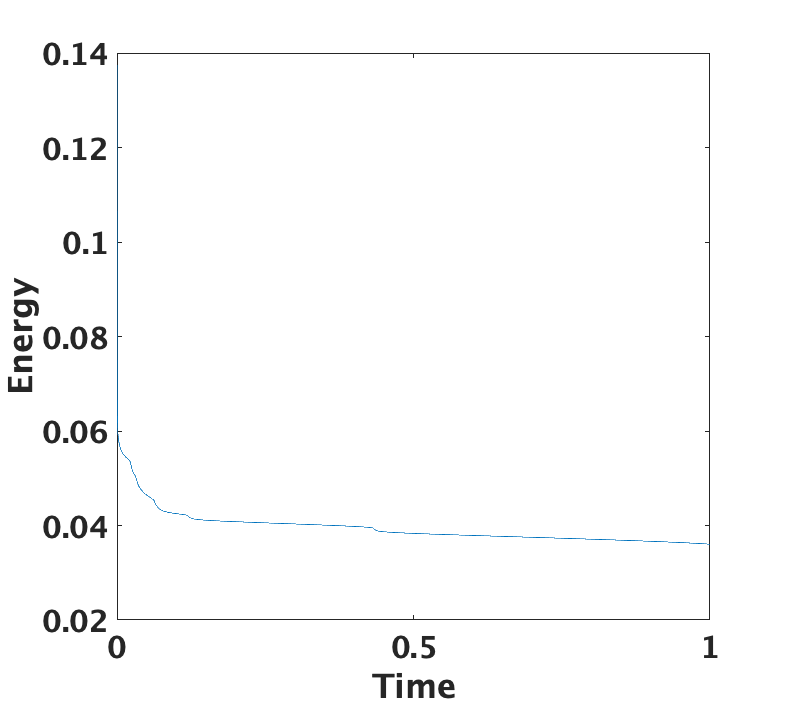}}
\subfigure[]{\includegraphics[width=0.4\textwidth]{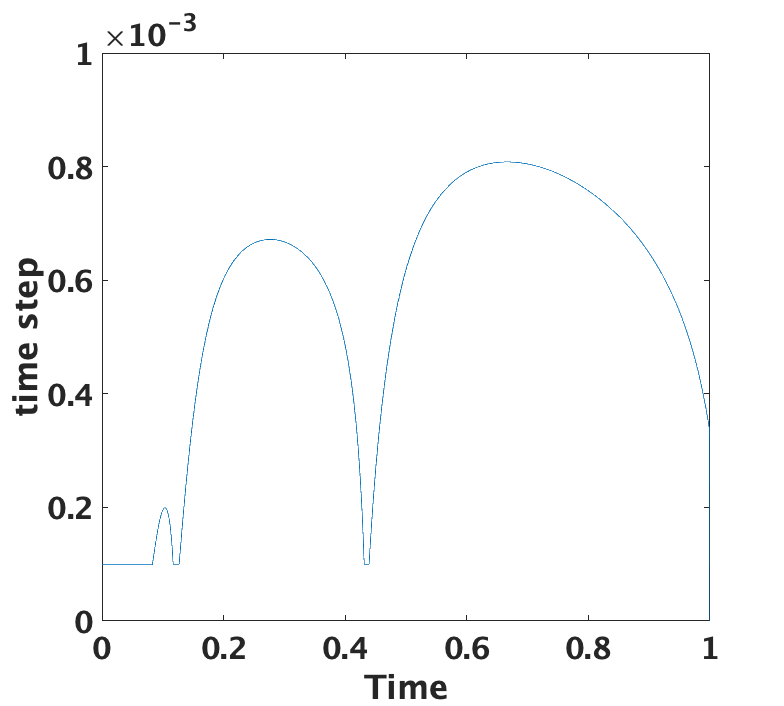}}
\caption{Energy evolution and time step sizes. In this figure, the energy evolution for the simulation in Figure \ref{fig:strip} is shown in (a); and the adaptive time step sizes used for the simulation in Figure \ref{fig:strip} is shown in (b).}
\label{fig:strip-energy}
\end{figure}

Also, to numerically verify the theorem results of \eqref{eq:mass-conservation}, we generate the graph showing the error of the total mess, i.e., the difference of
$$
\int_\Omega \phi(\x,t)d\x  -\int_\Omega \phi(\x,0) d\x,
$$
which is summarized in Figure \ref{fig:strip-mass}. We observe that the difference is in the order of $O(10^{-14})$, which is negligible. This is in strong agreement with \eqref{eq:mass-conservation}, i.e., our proposed numerical algorithm preserves the total mass of the phase variable $\phi$.

\begin{figure}[H]
\center
\includegraphics[width=0.65\textwidth]{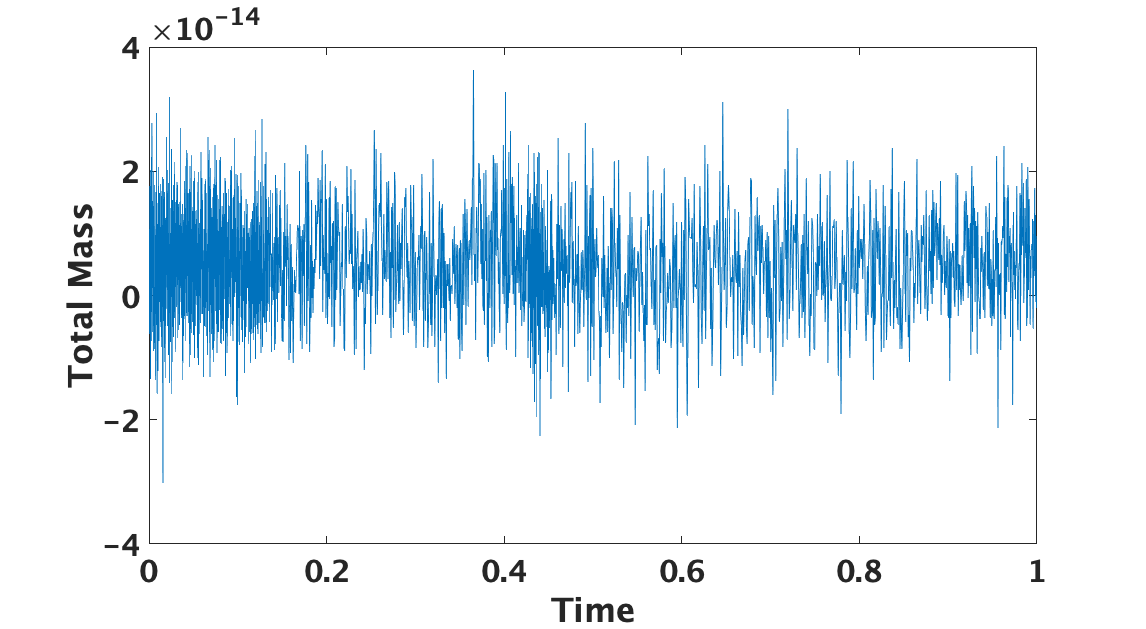}
\caption{Time evolution of the total volume error for the phase variable $\phi$. This figure shows the time evolution of $\int_\Omega \phi(\x,t)d\x  -\int_\Omega \phi(\x,0) d\x$. }
\label{fig:strip-mass}
\end{figure}

\section{Conclusion}
In this paper, we propose a new algorithm for solving the time-fractional Cahn-Hilliard equation. The resulted numerical scheme has several advantages. It allows non-uniform time step, such that adaptive time step sizes could be utilized to save computational time. It is unconditionally energy stable, assuring the stability of the numerical algorithm. The existence and uniqueness of the numerical solution are also verified theoretically. Several numerical tests are shown to confirm these advantages.

\section*{Acknowledgments}

The work of Jun Zhang is supported by the National Natural Science Foundation of China (No. 11901132), the China Scholarship Council (No. 201908525061) and the Science and Technology Program of Guizhou Province (No.[2020]1Y013). Jia Zhao would like to acknowledge the support by National Science Foundation, USA, with grant numner DMS-1816783.


\bibliography{ref}
\bibliographystyle{unsrt}

\end{document}